\documentclass[12pt,reqno]{article}

\usepackage[usenames]{color}
\usepackage{amssymb}
\usepackage{graphicx}
\usepackage{amscd}
\usepackage{xcolor}

\definecolor{webgreen}{rgb}{0,.5,0}
\definecolor{webbrown}{rgb}{.6,0,0}

\usepackage[colorlinks=true,
linkcolor=webgreen,
filecolor=webbrown,
citecolor=webgreen]{hyperref}

\usepackage{fullpage}
\usepackage{float}

\usepackage{graphics,amsmath,amssymb}
\usepackage{amsthm}
\usepackage{amsfonts}
\usepackage{latexsym}
\usepackage{epsf}

\usepackage{verbatim}
\usepackage{enumerate}
\usepackage{mathtools}

\usepackage{tikz}
\usepackage{tikz-3dplot}
\usetikzlibrary{calc}

\usepackage[blocks]{authblk}
\usetikzlibrary{graphs}

\usepackage{dcolumn}

\usepackage{lmodern}
\usepackage[utf8]{inputenc}

\usepackage[maxbibnames=10,style=alphabetic,maxalphanames=10]{biblatex}
\theoremstyle{plain}
\newtheorem{theorem}{Theorem}
\numberwithin{theorem}{section}
\newtheorem{corollary}[theorem]{Corollary}
\newtheorem{lemma}[theorem]{Lemma}
\newtheorem{proposition}[theorem]{Proposition}
\newtheorem{conjecture}[theorem]{Conjecture}

\newtheorem{prop}[theorem]{Proposition}
\newtheorem{question}[theorem]{Question}

\theoremstyle{definition}
\newtheorem{definition}[theorem]{Definition}
\newtheorem{example}[theorem]{Example}

\theoremstyle{remark}

\newtheorem*{remark*}{Remark}

\newcommand{\ol}{\overline}
\newcommand{\tn}{\textnormal}
\newcommand{\m}{\ol{\widetilde{m}}}
\newcommand{\tb}{\textbf}
\newcommand{\tcb}{\textcolor{blue}}
\newcommand{\se}{\subseteq}

\title{On graphs with equal and different Kromatic symmetric functions}

\author[1]{Laura Pierson}
\author[2]{Soham Samanta}
\affil[1]{\href{mailto:lcpierson73@gmail.com}{lcpierson73@gmail.com}}
\affil[2]{\href{mailto:soham2020sam@gmail.com}{soham2020sam@gmail.com}}
\begin{document}

\maketitle

\begin{abstract}

    The \emph{Kromatic symmetric function} (KSF) $\ol{X}_G$ of a graph $G$ is a $K$-analogue introduced by Crew, Pechenik, and Spirkl in \cite{crew2023kromatic} of Stanley's chromatic symmetric function (CSF) $X_G$. The KSF is known to distinguish some pairs of graphs with the same CSF. The first author showed in \cite{pierson2024counting} and \cite{new_p_expansion} that the number of copies in $G$ of certain induced subgraphs can be determined given $\ol{X}_G$, and conjectured that $\ol{X}_G$ distinguishes all graphs. We disprove that conjecture by finding four pairs of 8-vertex graphs with equal KSF, as well as giving several ways to use existing graph pairs with equal KSF to construct larger graph pairs that also have equal KSF. On the other hand, we show that many of the graph pairs from the constructions in \cite{orellana2014graphs} and \cite{aliste2021vertex} of graphs with the same CSF are distinguished by the KSF, thus also giving some new examples of cases where the KSF is a stronger invariant than the CSF.
\end{abstract}

\section{Introduction}

The chromatic symmetric function (CSF) is a well known generalization of the chromatic polynomial, introduced by Richard Stanley in \cite{stanley1995symmetric}. For weighted graphs, the CSF is defined as follows:
\begin{definition}[Stanley \cite{stanley1995symmetric}; Crew-Spirkl \cite{crew2020deletion} for weighted version]
    Let $G = (V,E)$ be a finite simple graph with vertex set $V$, together with a weight function $w:G\to \mathbb{P}$, where $\mathbb{P}$ is the set of positive integers. A \emph{\tb{\tcb{proper coloring}}} of $G$ is a function $\kappa : V \to \mathbb{P}$ such that $\kappa(u) \neq \kappa(v)$ whenever $uv \in E$. The \emph{\tb{\tcb{chromatic symmetric function (CSF)}}}  of the pair $(G,w)$, denoted $X_{(G,w)}(\mathbf{x})$ or simply $X_{(G,w)}$, is defined as:
    \[
    X_{(G,w)}(\mathbf{x}) := \sum_{\kappa} \prod_{v \in V} x_{\kappa(v)}^{w(v)},
    \]
    where the sum is over all proper colorings $\kappa$ of $G$, and $\mathbf{x} = (x_1, x_2, x_3, \dots)$ is a countably infinite set of commuting variables, one assigned to each color. When no weight function is specified, we will assume all vertices have weight 1 and write simply $X_G$.
\end{definition}

Stanley \cite{stanley1995symmetric} remarked that he was unsure if any two nonisomorphic tree exist with the same CSF. Since his paper was published, a number of authors have studied which pairs of graphs are distinguished by the CSF and worked toward trying to prove that it distinguishes all trees. The CSF is known to distinguish all trees on up to 29 vertices~\cite{heil2018algorithm}. It is also known to distinguish all graphs in several infinite families of trees ~\cite{morin2005caterpillars,aliste2014proper,martin2008distinguishing}, as well as in a family of unicyclic graphs called squids~\cite{martin2008distinguishing}. Various properties of a tree $T$ are also known to be computable from $X_T$, including the subtree polynomial, the path sequence, and the degree sequence~\cite{martin2008distinguishing}. For non-trees, Orellana and Scott \cite{orellana2014graphs} proved that $X_G$ also determines the girth of $G$ and the number of triangles. 

Although it is still open whether or not the CSF distinguishes all \emph{trees}, it is known that there do exist nonisomorphic \emph{graphs} with the same CSF. In particular, Orellana and Scott \cite{orellana2014graphs} and Aliste-Prieto, Crew, Spirkl, and Zamora \cite{aliste2021vertex} construct two different infinite families of pairs of graphs with equal CSF. We will investigate those constructions in \S \ref{sec:orellana-scott} and \S \ref{sec:split_graphs}, respectively.

Our main focus here is the following $K$-theoretic analogue of $X_G$, introduced by Crew, Pechenik, and Spirkl in \cite{crew2023kromatic}:

\begin{definition}[Crew-Pechenik-Spirkl ~\cite{crew2023kromatic}]
A \emph{proper set coloring} of $G$ is a function $\kappa : V \to 2^{\mathbb{P}} \setminus \{\varnothing\}$ such that $\kappa(u) \cap \kappa(v) = \varnothing$ whenever $uv \in E$, so each vertex receives a nonempty set of colors such that adjacent vertices receive nonoverlapping color sets. The \emph{\tb{\tcb{Kromatic symmetric function (KSF)}}} of $(G,w)$ is
\[
\overline{X}_{(G,w)} := \sum_{\kappa} \prod_{v \in V}  \left(\prod_{i \in \kappa(v)} x_i\right)^{w(v)}.
\]
That is, each monomial corresponds to a proper set coloring of $G$, with the exponent on each color variable being the sum of the weights of vertices receiving that color. Again, if no weight function is specified, we will assume all vertices have weight 1.
\end{definition}

Part of the motivation for defining $\ol{X}_G$ came from $K$-theory, the study of $K$-rings associated to topological spaces, which are a deformation of cohomology rings. It is unknown whether $X_G$ and $\ol{X}_G$ actually have a topological interpretation related to $K$-theory, but the idea was to help shed light on whether they might, and whether such an interpretation might be helpful for proving the Stanley-Stembridge conjecture (which has since been proven by Hikita \cite{hikita2024proof} using different methods). Whether or not such a topological interpretation exists, however, $\ol{X}_G$ is still interesting combinatorially. In particular, it has a natural interpretation in terms of Hopf algebras \cite{marberg2023kromatic} and it has nice expansion formulas using $K$-analogues of the Schur functions $s_\lambda$ \cite{crew2023kromatic}, the elementary symmetric functions $e_\lambda$ \cite{marberg2023kromatic}, and the power sum symmetric functions $p_\lambda$ \cite{new_p_expansion}.

Another reason $\ol{X}_G$ is interesting is that it contains more information about $G$ than $X_G$ does, since its lowest degree terms match $X_G$ but it also has additional higher degree terms. The authors of \cite{crew2023kromatic} gave some examples of graphs with the same $X_G$ but different $\ol{X}_G$, and asked what else can be learned about which graphs are distinguished by $\ol{X}_G$. \cite{new_p_expansion} gives the following characterization of the information contained in $\ol{X}_G$:

\begin{theorem}[Pierson \cite{new_p_expansion}, Corollary 1.6]\label{thm:ind_poly}
    Knowing $\ol{X}_G$ is equivalent to knowing the multiset of independence polynomials of induced subgraphs of $G$.
\end{theorem}

By the \emph{\tb{\tcb{independence polynomial}}}, we mean the polynomial $I_G(x)$ such that the coefficient of $x^k$ is the number of independent subsets of $V$ of size $k$, where an \emph{\tb{\tcb{independent set}}} (also called a \emph{\tb{\tcb{stable set}}}) is a subset of $V$ containing no two adjacent vertices.

In particular, for any graph $H$ which is \emph{\tb{\tcb{independence unique}}}, meaning no other nonisomorphic graph has the same independence polynomial, the number of induced subgraphs of $G$ isomorphic to $H$ is determined by $\ol{X}_G$. \cite{pierson2024counting} explicitly lists a number of these independence unique graphs on 4 and 5 vertices that can be counted by $\ol{X}_G$:

\begin{theorem}[Pierson~\cite{pierson2024counting}, Theorems 3 and 4]\label{thm:graph_list}
    The number of induced copies in $G$ of the following order 4 and order 5 graphs can be computed from $\overline{X}_G$:
    \begin{center}
        \includegraphics[scale=0.8]{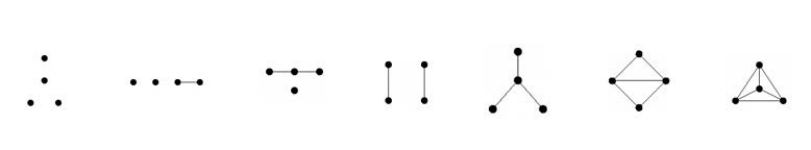}\\
        \includegraphics[scale=0.5]{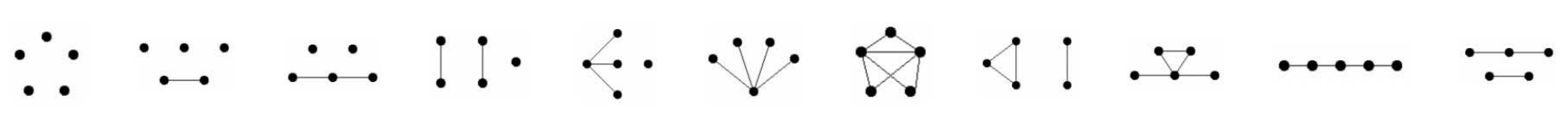}
    \end{center}
\end{theorem}
When counting induced copies of a subgraph, we consider two induced copies to be the same if they use the same set of vertices, so even if there is a nontrivial automorphism taking an induced subgraph to itself, it still only gets counted once.

\cite{pierson2024counting} went on to conjecture that the KSF actually distinguishes all graphs:

\begin{conjecture}[Pierson \cite{pierson2024counting}, Conjecture 2]\label{conjecture:distinguish}
There do not exist nonisomorphic graphs $G$ and $H$ with $\overline{X}_G = \overline{X}_H$.
\end{conjecture}

We disprove Conjecture \ref{conjecture:distinguish} in \S \ref{section:counterexamples} by showing four different pairs of nonisomorphic graphs on 8 vertices such that the two graphs in each pair have the same KSF as each other. We also give several ways to use these pairs to construct more graph pairs with equal KSF. 

\bigskip

Then in \S \ref{sec:orellana-scott} and \S \ref{sec:split_graphs}, we respectively investigate the Orellana-Scott \cite{orellana2014graphs} and Aliste-Prieto-Crew-Spirkl-Zamora \cite{aliste2021vertex} constructions of infinite families of pairs of graphs with equal CSF. A natural question is whether any of these graph pairs with equal CSF also have equal KSF. We have not found any such pairs, but we were instead able to show that given certain conditions, these graph pairs with equal CSF are distinguished by the KSF.

\bigskip

The code used for this project was written in Sage \cite{sage} with some help from ChatGPT \cite{chatgpt}, and can be found at this \href{https://github.com/Soham2020sam/KromaticSymmetricFunction/}{Github link}.

\bigskip

In the process of finding independence unique graphs that we could use when applying Theorem \ref{thm:ind_poly}, we counted the number of independence unique graphs on small numbers of vertices. We posted those counts as Sequence A385864 on OEIS \cite{oeis}, which can be found at \href{https://oeis.org/A385864}{https://oeis.org/A385864}. The known terms so far are $$1, 2, 4, 7, 13, 24, 53, 109, 284, 746, 2416, 8804.$$


\section{Graphs with equal Kromatic symmetric function}\label{section:counterexamples}

Based on our code, we found that in each of the four pairs of 8-vertex graphs shown below, the two graphs are nonisomorphic but have the same KSF. These graphs are thus counterexamples to Conjecture \ref{conjecture:distinguish}, and our code also shows that they are the smallest counterexamples that exist. Note that in the first pair, the graphs can be made isomorphic by removing a single edge from each graph, and in each other pair, they can be made isomorphic by removing two edges from each graph.
\begin{center}
    \includegraphics[width=5cm]{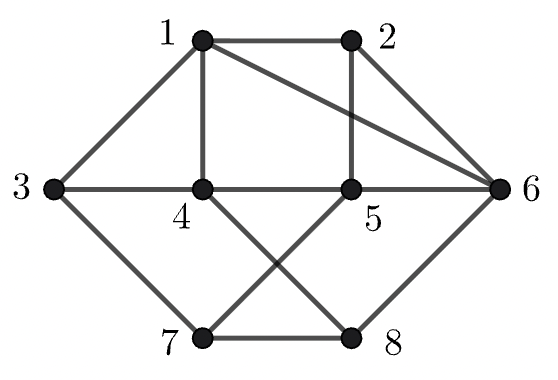}
    \hspace{0.5cm}
    \includegraphics[width=5cm]{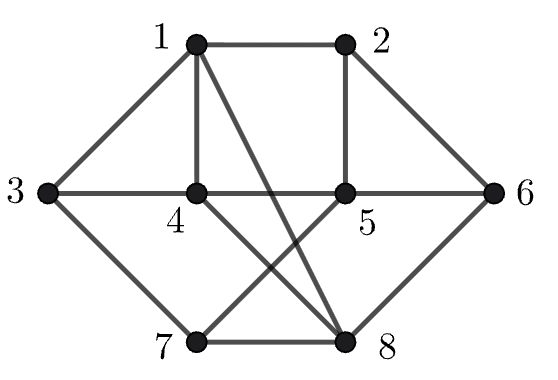} \\
    \vspace{0.5cm}
    \includegraphics[width=5cm]{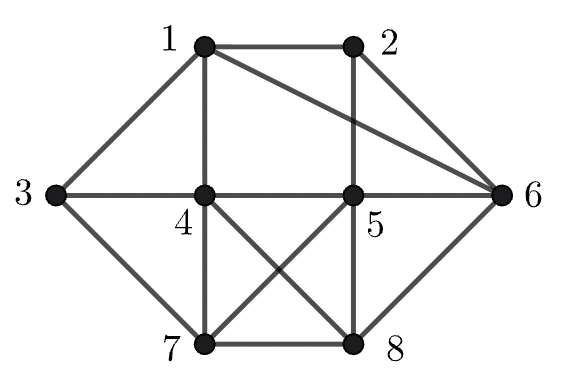}
    \hspace{0.5cm}
    \includegraphics[width=5cm]{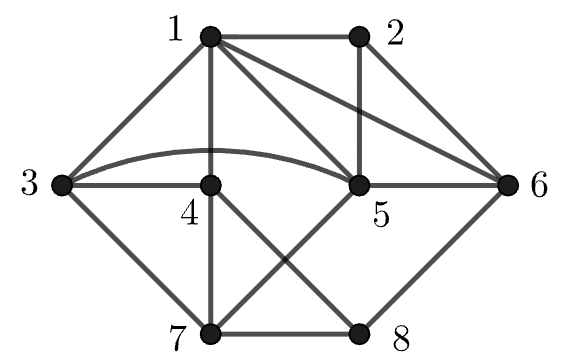} \\
    \vspace{0.5cm}
    \includegraphics[width=5cm]{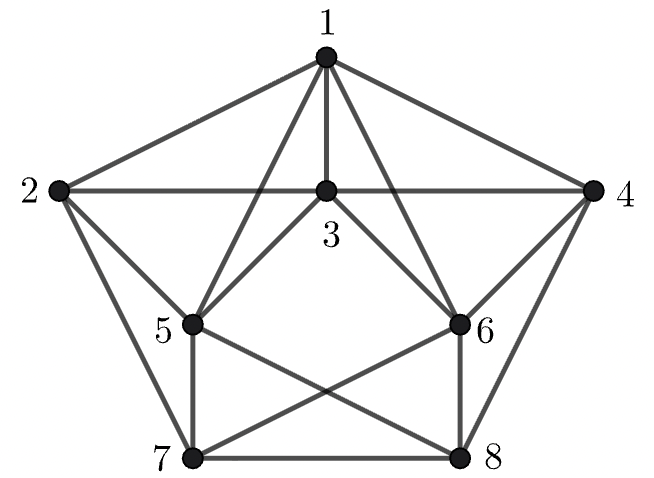}
    \hspace{0.5cm}
    \includegraphics[width=5cm]{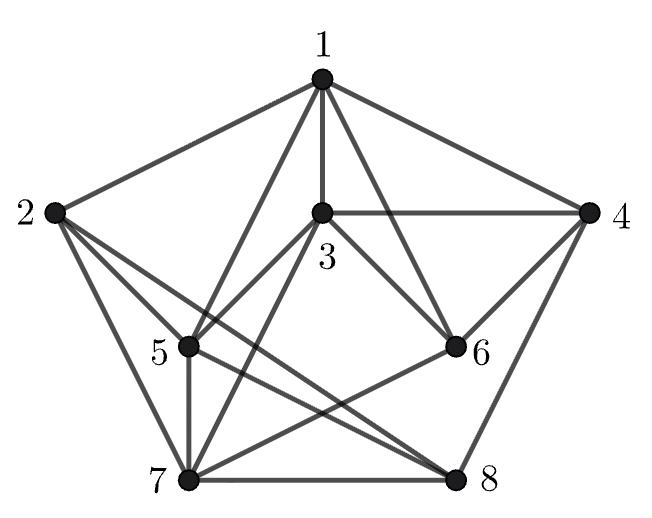} \\
    \vspace{0.5cm}
    \includegraphics[width=5cm]{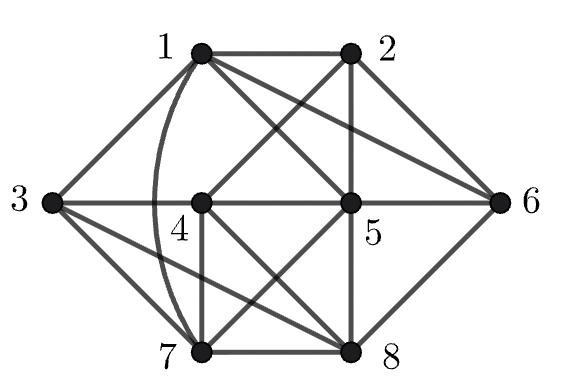}
    \hspace{0.5cm}
    \includegraphics[width=5cm]{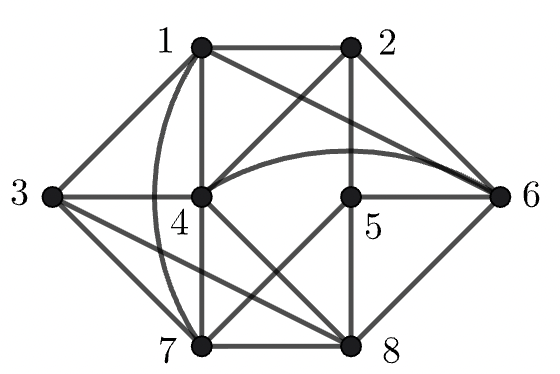} \\
\end{center}

We now give several ways to use these examples to construct more pairs of nonisomorphic graphs with equal KSF. The simplest way is by taking disjoint unions: 

\begin{prop}
    If $\ol{X}_{G_1} = \ol{X}_{G_2}$ and $\ol{X}_{H_1} = \ol{X}_{H_2},$ then $\ol{X}_{G_1\sqcup H_1} = \ol{X}_{G_2\sqcup H_2}.$
\end{prop}

\begin{proof}
    Note that for any graphs $G$ and $H,$ a proper set coloring of $G\sqcup H$ corresponds to taking any proper set coloring of $G$ together with any proper set coloring of $H$. The monomials in $\ol{X}_G$ and $\ol{X}_H$ corresponding to those two proper set colorings multiply to give a monomial in $\ol{X}_{G\sqcup H}$ corresponding to the resulting proper set coloring of $G\sqcup H$. It follows that $\ol{X}_{G\sqcup H} = \ol{X}_G\ol{X}_H$. Thus, given our assumptions, we have $$\ol{X}_{G_1\sqcup H_1} = \ol{X}_{G_1}\ol{X}_{H_1} = \ol{X}_{G_2}\ol{X}_{H_2} = \ol{X}_{G_2\sqcup H_2}.$$
\end{proof}

The second way is by taking joins, where the \emph{\tb{\tcb{join}}} $G\odot H$ is the graph formed by connecting all vertices of $G$ to all vertices of $H$:

\begin{prop}\label{prop:join_equal_ksf}
    If $\ol{X}_{G_1} = \ol{X}_{G_2}$ and $\ol{X}_{H_1}=\ol{X}_{H_2}$, then $\ol{X}_{G_1\odot H_1} = \ol{X}_{G_2\odot H_2}.$
\end{prop}

We will give two different proofs. One proof is using the characterization of the KSF in terms of independence polynomials from Theorem \ref{thm:ind_poly}:

\begin{proof}[Proof 1.]
     Since $\ol{X}_{G_1} = \ol{X}_{G_2}$, for each $G_1'$ that is an induced subgraph in $G_1$, there exists a corresponding induced subgraph in $G_2'$ of $G_2$ such that $I_{G_1'} = I_{G_2'}$. We similarly have that for each $H_1'$ that is an induced subgraph of $H_1$, there exists a corresponding induced subgraph in $H_2'$ of $H_2$ such that $I_{H_1'} = I_{H_2'}$. Each induced subgraph of $G_1 \odot H_1$ can be written as $G_1' \odot H_1',$ where $G_1'$ is an induced subgraph of $G_1$ and $H_1'$ is an induced subgraph of $H_1$. Then $$I_{G_1' \odot H_1'} = I_{G_1'} + I_{H_1'} - 1,$$ because each independent set in $G_1' \odot H_1'$ is either only in $G_1'$ or only in $H_1'$ or is empty, and we subtract 1 because we overcount the empty subset. Thus, we can pair the induced subgraph $G_1'\odot H_1'$ of $G_1\odot H_1$ with the induced subgraph $G_1'\dot H_1'$ of $G_2 \odot H_2$, and we have $I_{G_2' \odot H_2'} = I_{G_1' \odot H_1'}$. Because of this pairing, $G_1\odot H_1$ and $G_2\odot H_2$ have the same multiset of independence polynomials of induced subgraphs, hence the same KSF by Theorem \ref{thm:ind_poly}.
\end{proof}

For the second proof, we will introduce some additional notation, which will also be useful for our other constructions of more graph pairs with equal KSF. First we review some background on symmetric functions.

A \emph{\tb{\tcb{symmetric function}}} is a power series of bounded degree in a countably infinite variable set that stays the same under any permutation of the variables. We write $\Lambda$ for the ring of symmetric functions, and $\ol{\Lambda}$ for its \emph{\tb{\tcb{completion}}}, which allows power series of unbounded degree. Thus, $X_G$ lives in $\Lambda\subseteq \ol{\Lambda},$ and $\ol{X}_G$ technically lives not in $\Lambda$ but in $\ol{\Lambda}.$

A \emph{\tb{\tcb{partition}}} $\lambda = \lambda_1\dots\lambda_\ell$ is a nondecreasing sequence $\lambda_1\ge \dots \ge \lambda_\ell$ of positive integers, called the \emph{\tb{\tcb{parts}}}. The classic vector space bases for $\Lambda$ are all indexed by partitions. In particular, the \emph{\tb{\tcb{monomial symmetric function}}} $m_\lambda$ is the sum of all monomials of the form $x_{i_1}^{\lambda_1}\dots x_{i_\ell}^{\lambda_\ell}$ such that $i_1,\dots,i_\ell$ are all distinct, and the $m_\lambda$'s form a basis for $\Lambda$. A slight modification of this basis is the \emph{\tb{\tcb{augmented monomial symmetric functions}}} $$\widetilde{m}_\lambda := m_\lambda \cdot \prod_{i\ge 1} (r_i(\lambda))!,$$ where $r_i(\lambda)$ is the number of parts of $\lambda$ of size $i$. The $\widetilde{m}$'s are useful in the context of CSFs because $\widetilde{m}_\lambda$ is the CSF $X_{K_\lambda}$ of the weighted complete graph $K_\lambda:= (K_{\ell},\lambda)$ on $\ell$ vertices with vertex weights equal to the parts $\lambda_1,\dots,\lambda_\ell$ of $\lambda$. The $\widetilde{m}$'s are also useful because for any weighted graph $(G,w)$, the expansion of $X_{(G,w)}$ in the $\widetilde{m}$-basis has a nice interpretation:

\begin{prop}[Stanley \cite{stanley1995symmetric}, Theorem 2.5; Crew-Spirkl \cite{crew2020deletion}, Lemma 1]\label{prop:normal_m}
    $[\widetilde{m}_\lambda]X_{(G,w)}$ is the number of ways to partition the vertices of $G$ into stable sets such that when the vertex weights are added within each stable set, the sums are the parts of $\lambda.$
\end{prop}

Inspired by this, the authors of \cite{crew2023kromatic} defined the following $K$-analogue of the $\widetilde{m}$-basis:

\begin{definition}[Crew-Pechenik-Spirkl \cite{crew2023kromatic}]
    The \emph{\tb{\tcb{$K$-theoretic augmented monomial symmetric function}}} $\m_\lambda$ is $$\m_\lambda := \ol{X}_{K_\lambda}.$$ 
\end{definition}

They then gave a $K$-analogue of the $\widetilde{m}$-expansion formula for CSFs, where essentially the only difference is that the stable sets are allowed to overlap:

\begin{definition}[Crew-Pechenik-Spirkl \cite{crew2023kromatic}]
    Define a \emph{\tb{\tcb{stable set cover (SSC)}}} to be a set of stable sets such that each vertex in the graph is covered by at least one stable set. Write ${\textsf{SSC}}(G)$ for the set of all SSCs of $G$. For each $C\in {\textsf{SSC}}(G)$, write $\lambda(C)$ for the partition whose parts are the sums of the vertex weights within each stable set in $C$.
\end{definition}

Then their $\m$-expansion for $\ol{X}_G$ is as follows:

\begin{prop}[Crew-Pechenik-Spirkl \cite{crew2023kromatic}, Proposition 3.4]\label{prop:crewssc}
    For any vertex-weighted graph $(G, w)$, we have
    $$\ol{X}_{(G, w)} = \sum_{C \in {\textsf{\textup{SSC}}}(G)} \ol{\widetilde{m}}_{\lambda(C)}.$$
\end{prop}

We will now introduce some additional notation due to Tsujie \cite{tsujie2018chromatic} that will be helpful for describing $\m$-expansions of the KSFs of certain graphs:

\begin{definition}[Tsujie \cite{tsujie2018chromatic}]
    Define an alternative multiplication $\odot$ on symmetric functions by $$\widetilde{m}_\lambda \odot \widetilde{m}_\mu := \widetilde{m}_{\lambda \sqcup \mu},$$ where $\lambda\sqcup \mu$ is the partition each of whose parts is either a part of $\lambda$ or a part of $\mu$. Write $\widetilde{\Lambda}$ for the ring of symmetric functions with the usual addition and scalar multiplication and this modified multiplication, and $\ol{\widetilde{\Lambda}}$ for the completion of $\widetilde{\Lambda}.$
\end{definition}

\begin{example}
    We have $3211 \sqcup 42 = 432211,$ so $\widetilde{m}_{3211}\odot \widetilde{m}_{42} = \widetilde{m}_{432211}.$
\end{example}

Tsujie proves that the CSFs of simple graphs multiply over \emph{joins} under this $\odot$ multiplication (as opposed to multiplying over \emph{disjoint unions} under the normal multiplication), and we observe here that the same holds for weighted graphs. Namely, for weighted graphs $(G,w_G)$ and $(H,w_H)$, define the join $(G,w_G)\odot (H,w_H)$ to be the weighted graph formed by connecting all vertices of $G$ to all vertices of $H$, and taking the weight of each vertex to be its weight in whichever of $G$ or $H$ it comes from. Then we have:

\begin{proposition}[cf. Tsujie \cite{tsujie2018chromatic}, Lemma 2.9]\label{prop:join_multiply}
    For any weighted graphs $(G,w_G),$ $(H,w_H),$ $$X_{(G,w_G)\odot(H,w_H)} = X_{(G,w_G)}\odot X_{(H,w_H)}.$$
\end{proposition}

\begin{proof}
    The proof is the same as for unweighted graphs. No stable set can contain both a vertex from $G$ and a vertex from $H$, so every stable set in $G\odot H$ is either a stable set in $G$ or a stable set in $H$. Thus, partitioning $V(G\odot H)$ into stable sets is equivalent to partitioning $V(G)$ into stable sets and also partitioning $V(H)$ into stable sets, and then taking the full collection of stable sets to be the disjoint union of those two collections of stable sets. The associated partition is then $\lambda \sqcup \mu$, where $\lambda$ is the partition whose parts are the total weights of the stable sets in $G$, and $\mu$ is the partition whose parts are the total weights of the stable sets in $H$. Thus, by Proposition \ref{prop:normal_m}, each $\widetilde{m}_{\nu}$ term in $X_{(G,w_G)\odot (H,w_H)}$ can be uniquely rewritten as $\widetilde{m}_{\lambda\sqcup \nu} = \widetilde{m}_\lambda \odot \widetilde{m}_\mu,$ where $\widetilde{m}_\lambda$ is a term in $X_{(G,w_G)}$ and $\widetilde{m}_\mu$ is a term in $X_{(H,w_H)}$. The result then follows from the distributive property of $\odot$ over addition.
\end{proof}

Next, we observe that if we work in the completion $\ol{\widetilde{\Lambda}}$, KSFs of weighted graphs also multiply over joins under the same $\odot$ multiplication:

\begin{prop}\label{prop:K_join_multiply}
    For any weighted graphs $(G,w_G),$ $(H,w_H)$, $$\ol{X}_{(G,w_G)\odot (H,w_H)} = \ol{X}_{(G,w_G)} \odot \ol{X}_{(H,w_H)}.$$
\end{prop}

\begin{proof}
    As noted in \cite{crew2023kromatic}, we can rewrite the KSF $\ol{X}_{(G,w_G)}$ as a linear combination of CSFs of clan graphs of $G$, namely, $$\ol{X}_{(G,w_G)} = \sum_{\alpha\vDash |V(G)|}\frac1{\alpha!}X_{C_\alpha(G,w_G)},$$ where $\alpha=\alpha_1\dots\alpha_{|V(G)|}$ is a \emph{\tb{\tcb{strict composition}}} of $|V(G)|$ (i.e. a sequence of $|V(G)|$ positive integers), $\alpha! := \alpha_1!\dots \alpha_{|V(G)|}!$, and $C_\alpha(G,w_G)$ is the \emph{\tb{\tcb{$\alpha$-clan graph}}} formed by blowing up the $i$th vertex $v_i$ of $G$ into a clique of $\alpha_i$ vertices all with the same weight as $v_i$, such that two vertices in $C_\alpha(G,w_G)$ are adjacent if they come from either the same or adjacent vertices in $G$. Then using the fact that $\odot$ distributes over addition together with Proposition \ref{prop:join_multiply},
    \begin{align*}
        \ol{X}_{(G,w_G)}\odot \ol{X}_{(H,w_H)} &= \left(\sum_\alpha \frac1{\alpha!}X_{C_\alpha(G,w_G)}\right)\odot \left(\sum_\beta \frac1{\beta!}X_{C_\beta(H,w_H)}\right) \\
        &= \sum_{\alpha,\beta} \frac1{\alpha!\beta!}X_{C_\alpha(G,w_G)}\odot X_{C_\beta(H,w_H)} \\
        &= \sum_{\alpha,\beta} \frac1{\alpha!\beta!}X_{C_\alpha(G,w_G)\odot C_\beta(H,w_H)}
    \end{align*}
    But the join of the $\alpha$-clan graph of $G$ and the $\beta$-clan graph of $H$ is the same thing as the $\gamma$-clan graph of $G\odot H$, where $\gamma = \alpha \sqcup \beta := \alpha_1\dots \alpha_{|V(G)|}\beta_1\dots \beta_{|V(H)|}$ is the composition formed by concatenating $\alpha$ and $\beta$. Then $\gamma! = \alpha!\beta!$, so $$\ol{X}_{(G,w_G)}\odot \ol{X}_{(H,w_H)} = \sum_\gamma \frac1{\gamma!}X_{C_\gamma((G,w_H)\odot(H,w_H))} = \ol{X}_{(G,w_G)\odot(H,w_H)},$$ as claimed.
\end{proof}

Note that this immediately gives another proof for Proposition \ref{prop:join_equal_ksf}:

\begin{proof}[Proof 2 of Proposition \ref{prop:join_equal_ksf}]
    Assuming $\ol{X}_{G_{1}}= \ol{X}_{G_{2}}$ and $\ol{X}_{H_{1}} = \ol{X}_{H_2}$, we get $$\ol{X}_{G_{1} \odot H_{1}} = \ol{X}_{G_1}\odot \ol{X}_{G_2} = \ol{X}_{H_1}\odot \ol{X}_{H_2} = \ol{X}_{G_2 \odot H_2}.$$
\end{proof}

Note also that a special case of Proposition \ref{prop:K_join_multiply} is that the $\m$'s multiply in $\ol{\widetilde{\Lambda}}$, since $\m_\lambda = \ol{X}_{K_\lambda}$, and $K_{\lambda \sqcup \mu} = K_\lambda \odot K_\mu$:

\begin{corollary}\label{cor:m_multiply}
    For any partitions $\lambda$ and $\mu$, $$\m_{\lambda\sqcup \mu} = \m_\lambda \odot \m_\mu.$$
\end{corollary}

Now, we give another construction for building larger graph pairs with equal KSF out of smaller such pairs. We assume here that our graphs are unweighted, although a similar statement should hold for weighted graphs.

\begin{theorem}\label{thm:equalKotherV}
    Assume graphs $G_1, G_2, H_1, H_2$ satisfy $\ol{X}_{G_{1}} = \ol{X}_{G_{2}}$ and $\ol{X}_{H_{1}} = \ol{X}_{H_{2}}$. Suppose we also have vertices $v_1$ in $G_1$ and $v_2$ in $G_2$ such that $\ol{X}_{G_{1} - v_1} = \ol{X}_{G_{2} - v_2}$. Let $G_{1}'$ be the graph formed by attaching $H_1$ to all vertices of $G_1$ except $v_1$, and let $G_2'$ be the graph formed by attaching $H_2$ to all vertices of $G_2$ except $v_2$. Then $\ol{X}_{G_{1}'} = \ol{X}_{G_{2}'}$.
\end{theorem}

\begin{proof}
To prove Theorem~\ref{thm:equalKotherV}, we first define the following function:
\begin{definition}
    Let $$f(v, G):= \sum_{C\in {\textsf{SSC}}(v,G)}\m_{\lambda(C)},$$ where ${\textsf{SSC}}(v,G) \subset {\textsf{SSC}}(G)$ is the set of SSCs $C$ of $G$ such that $\{v\}\not\in C$ (i.e. $v$ is covered only by stable sets of size greater than 1).
\end{definition}

\begin{lemma}\label{lemma:f(v, G)}
    $f(v, G)$ can be determined from $\ol{X}_{G}$ and $\ol{X}_{G - v}$.
\end{lemma}
\begin{proof}
    We claim that 
    \begin{equation}\label{eq:f(v,G)}
        \ol{X}_{G} = \ol{X}_{G - v} \odot \ol{\widetilde{m}}_1 + f(v, G) \odot (\ol{\widetilde{m}}_1 + 1).
    \end{equation} To see this, consider the $\m$-expansion of $\ol{X}_G$ from Proposition \ref{prop:crewssc}, and note that every $C\in {\textsf{SSC}}(G)$ falls under one of three cases:
    \begin{itemize}
        \item \tb{Case 1:} $\{v\}\in C$, and $v$ is in no other stable set in $C$. In this case, we can uniquely write $C = C'\sqcup \{\{v\}\}$ with $C'\in {\textsf{SSC}}(G-v)$. Then $\lambda(C) = \lambda(C')\sqcup 1,$ so $\m_{\lambda(C)}=\m_{\lambda(C')}\odot \m_1$. Thus, each term in the $\m$-expansion of $\ol{X}_G$ that falls under this case is the $\odot$ product of $\m_1$ and a term in the $\m$-expansion of $\ol{X}_{G-v}$, so the sum of the $\m$-terms for this case is $\ol{X}_{G-v}\odot \m_1.$
        \item \tb{Case 2:} $\{v\}\not\in C$. The sum of the $\m$-terms for this case is $f(v,G)$ by definition.
        \item \tb{Case 3:} $\{v\}\in C$, and $v$ is also in another stable set in $C$. Each SSC $C$ in this case is of the form $C = C'\sqcup \{\{v\}\}$ with $C'\in {\textsf{SSC}}(v,G).$ Then $\lambda(C) = \lambda(C')\sqcup 1,$ so $\m_{\lambda(C)}=\m_{\lambda(C')}\odot \m_1.$ Thus, the sum of $\m$-terms for this case is $f(v,G)\odot \m_1$ by a similar argument to Case 1.
    \end{itemize}
    Combining these cases proves (\ref{eq:f(v,G)}). Now in $\ol{\widetilde{\Lambda}},$ we can write $$(1+\m_1)\odot(1-\m_1 + \m_{11}-\m_{111}+\dots)=1,$$ since $\m_{1^k}$ is the $k$th power of $\m_1$ under $\odot$ multiplication. Thus, we can solve for $f(v,G)$ as $$f(v,G) = (\ol{X}_G - \ol{X}_{G-v}\odot \m_1)\odot (1-\m_1 + \m_{11}-\m_{111}+\dots),$$ which lets us calculate $f(v,G)$ given $\ol{X}_G$ and $\ol{X}_{G-v}$.
\end{proof}

\begin{lemma}\label{lemma:graphG'}
    If we form $G'$ by attaching $H$ to all vertices of $G$ except $v$, then
    $$\ol{X}_{G'} = (f(v,H\sqcup v)+\ol{X}_H)\odot (f(v,G)+\ol{X}_{G-v})\odot (\m_1+1)-\ol{X}_H\odot \ol{X}_{G-v}$$
\end{lemma}

\begin{proof}
    Note that no vertex of $G-v$ can be in a stable set that also includes vertices from $H$, so we can essentially split each stable set cover $C\in {\textsf{SSC}}(G)$ into three parts:
    \begin{enumerate}
        \item Stable sets involving vertices from $H$ (and possibly $v$ in addition to those vertices). 
        \item Stable sets involving vertices from $G-v$ (and again possibly $v$).
        \item The stable set $\{v\}$ by itself, if it is used.
    \end{enumerate}
    We can essentially choose the stable sets used in each of these three parts independently of each other and then take the full SSC to be the disjoint union of those collections of stable sets, except that we will need to subtract cases where none of the three parts includes a stable set covering $v.$ We can thus separately find the generating series for the possible $\m$-terms enumerating the stable set sizes in each part, then take the $\odot$ product of those three generating series, and then subtract the $\m$-terms corresponding to cases where $v$ does not get covered by any of the stable sets. We consider these three generating series one by one:
    \begin{enumerate}
        \item For stable sets covering $H$, all vertices of $H$ must be covered, and $v$ can either be covered in a stable set that also includes other vertices of $H$, or it could be not covered at all. The former cases corresponds to the $\m$-terms in $f(v,H\sqcup v)$, and the latter to the $\m$-terms in $\ol{X}_G$, so combining those two cases, the generating series is $f(v,H\sqcup v)+\ol{X}_H.$
        
        \item For stable sets covering $G-v$, by similar reasoning, if $v$ is also covered by one of the stable sets, we get an $\m$-term from $f(v,G)$, and if $v$ is not covered, we get an $\m$-term from $\ol{X}_{G-v}$, so the sum of all the possible $\m$-terms is $f(v,G)+\ol{X}_{G-v}.$
        
        \item Finally, we can either include the singleton $\{v\}$ or not. In the former case, we get an $\m_1$ from $\{v\}$, and in the latter case we simply get 1 from the empty set.
    \end{enumerate}
    After taking the $\odot$ product of these three parts (which corresponds to taking the disjoint union of the associated stable set collections), the terms we need to subtract are the ones where in all three cases, the stable sets used do not cover $v.$ This means we need to take an $\m$-term from $\ol{X}_H$ in the first factor, an $\m$-term from $\ol{X}_{G-v}$ in the second factor, and the 1 from the third factor. Subtracting those cases not covering $v$ from the $\odot$ product of the generating series for the three parts proves Lemma \ref{lemma:graphG'}.
\end{proof}

Now to complete the proof of Theorem \ref{thm:equalKotherV}, it suffices to show that the expression in Lemma \ref{lemma:graphG'} is the same for $G_1'$ as for $G_2'$. Almost all the terms can be immediately matched up based on our assumptions and Lemma \ref{lemma:f(v, G)}. The one part to check is that $f(v_1,H_1\sqcup v_1) = f(v_2,H_2\sqcup v_2).$ Since $\ol{X}_{H_1}=\ol{X}_{H_2}$, this follows as long as in general, $f(v,H\sqcup v)$ can be determined from $\ol{X}_H$, which it can because by Lemma \ref{lemma:f(v, G)}, it depends only on $\ol{X}_H$ and $\ol{X}_{H\sqcup v}$, and $\ol{X}_{H\sqcup v}$ can be determined from $\ol{X}_H$ since $\ol{X}_{H\sqcup v} = \ol{X}_H \ol{X}_{\{v\}}=\ol{X}_G\m_1.$ (Note that we are using normal multiplication in this last equation, not $\odot$ multiplication.)
\end{proof}

\begin{example}
    For each of the four graph pairs $(G_1,G_2)$ with equal KSFs from the start of this section, we list all vertex pairs $(v_1,v_2)$ to which Theorem \ref{thm:equalKotherV} applies, i.e. all pairs $(v_1,v_2)$ such that $\ol{X}_{G_1-v_1}=\ol{X}_{G_2-v_2}.$ Note that since these graphs have $8$ vertices, which we found to be the minimum number of vertices such that two nonisomorphic graphs can have the same KSF, the 7-vertex graphs $G_1-v_1$ and $G_2-v_2$ must actually be isomorphic in every case in order to have the same KSF.
    \begin{itemize}
        \item \tb{1st graph pair:} The vertex pairs $(1,1)$ and $(3,2)$ work.
        \item \tb{2nd graph pair:} The vertex pairs $(5, 1),$ $(5,5),$ $(6,4),$ and $(6,7)$ work.
        \item \tb{3rd graph pair:} No vertex pairs $(v_1,v_2)$ work.
        \item \tb{4th graph pair:} The vertex pairs $(2,2),$ $(4,1)$, $(6,2),$ and $(8,1)$ work.
    \end{itemize}
\end{example}

Now we give a generalization of this construction, where we connect the new graphs $H_1$ and $H_2$ to all but $k$ vertices of $G$ and $G_2$, such that those $k$ vertices form cliques $C_1\subseteq G_1$ and $C_2 \subseteq G_2$ to form $G_1'$ and $G_2'$:

\begin{theorem}\label{cliqueKvertex}
    Consider graphs $G_1, H_1, G_2, H_2$. Assume that we have a set of vertices in $G_1$ that form a clique $C_1$ and we have a set of vertices in $G_2$ that form a clique $C_2\se G_2$ with $|C_2|=|C_1|$. Let $G_1'$ be the graph formed by connecting all vertices in $H_1$ to all vertices in $G_1$ that are not part of $C_1$. Similarly, define $G_2'$. Then, if for all $0 \le j \le |C_1|$, $$\sum_{S \subseteq C_1, |S| = j} \ol{X}_{G_1 - S} = \sum_{S \subseteq C_2, |S| = j} \ol{X}_{G_2 - S}$$ and we also have $\ol{X}_{H_1} = \ol{X}_{H_2}$, then we have $\ol{X}_{G_1'} = \ol{X}_{G_2'}$.
\end{theorem}

\begin{proof}
We first define a generalized version of the function $f$ that we previously defined:
\begin{definition}
    Define $f(C, G)$ be the generating series for ways to cover $G$ with stable sets such that all vertices in $C$ are only in stable sets of size greater than $1$.
\end{definition}

We can write $\ol{X}_{G_1'}$ and $\ol{X}_{G_2'}$ in terms of this function $f$ using the following lemma:

\begin{lemma}\label{lem:KSF_from_f(C,G)}
    Assume $G'$ is formed from $G$ by connecting all vertices of $H$ to the vertices of $G$ not in a clique $C \se V(G)$. Then $$\ol{X}_{G'} = \sum_{i=0}^{|C|}\left(\sum_{S'\se C,|S'|=i}f(C-S',G-S')\right)\odot \left(\sum_{j=i}^{|C|}\binom{|C|-i}{j-i}\m_{1^j}\ol{X}_H\right).$$
\end{lemma}

\begin{proof}
    We use casework on the set $S\se C$ of vertices in $C$ covered by at least one stable set involving no other vertices of $G$ (but optionally also by a stable set involving things in $C$), so every vertex in $S$ is either covered as a singleton or in a stable set involving vertices in $H$. Assume $|S|=j$. Let $S'\se S$ be the set of vertices in $S$ that are not covered by a stable set involving other vertices in $G$, and assume $|S'|=i\le j.$ Thus, all vertices in $C-S'$ are covered only in stable sets also involving other vertices in $G$. In fact, these stable sets can only involve vertices in $G-S'$, since all vertices in $S'$ are adjacent to all vertices in $C-S'$ (as $C$ is a clique). We can thus build an SSC for a given choice of $S$ and $S'$ as follows:
    \begin{itemize}
        \item We can first cover the vertices of $G-S'$ in such a way that all vertices in $C-S'$ are in stable sets of size greater than 1, and the generating for ways to do that is $f(C-S',G-S')$ by definition.
        \item Then we need to cover $S\sqcup H$ with no restrictions. The generating series for ways to do that is $\ol{X}_{S\sqcup H} = \ol{X}_S\ol{X}_H=\m_{1^j}\ol{X}_H$, since the vertices in $S$ form a clique, so the only way to cover them is with stable sets of size 1, giving $\ol{X}_S = \m_{1^j}$.
    \end{itemize}
    Thus, for a particular choice of $S$ and $S'$ with $|S|=j$, the generating series is $$f(C-S',G-S')\odot (\m_{1^j}\ol{X}_H).$$ Now we need to sum over all choices of $S$ and $S'$. Note that our expression above depends on $S'$, but not on $S$ as long as $|S|=j$. Thus, we can group our terms by first choosing $S'$ and then choosing $S$. For a given $S'$, we can choose any $S$ such that $S'\se S\se C$. Thus, if $|S'|=i$ and we want $|S|=j$, there are $\binom{|C|-i}{j-i}$ choices for $S$, because we need to choose $j-i$ more elements to be in $S$ out of the $|C|-i$ remaining elements in $C$, and for each such $S$, we get an $f(C-S',G-S')\odot (\m_{1^j}\ol{X}_H)$ term. Summing over the possible subsets $S'$ grouped by the size $|S'|=i$ gives the expression in Lemma \ref{lem:KSF_from_f(C,G)}.
\end{proof}

Now we show that these $f$ terms can be written in terms of the sums of KSFs showing up in Theorem \ref{cliqueKvertex}:

\begin{lemma}\label{lem:f's_using_X_(G-S)'s}
    For each $0\le i\le |C|,$ the sum $$\sum_{S'\se C,|S'|=i}f(C-S',G-S')$$ can be written in terms of sums of the form $\sum_{S\se C,|S|=j}\ol{X}_{G-S}.$
\end{lemma}

\begin{proof}
    We use ``reverse induction" on $i$, going from $i=|C|$ down to $i=0$. \\
    \\
    \tb{Base case:} If $i=|C|$, the only possibility is $S'=C$, so our sum becomes $$\sum_{S'\se C,|S'|=|C|}f(C-S',G-S')=f(\varnothing,G-C)=\ol{X}_{G-C}=\sum_{S\se C,|S|=|C|}\ol{X}_{G-S}.$$
    \tb{Inductive step:} Assume the statement holds for $i+1,i+2,\dots,|C|.$ We want to show that it also holds for $i$. For each $S'\se C$ with $|S'|=i$, plugging in $G-S'$ in place of $G$, $C-S'$ in place of $C$, and $H=\varnothing$ in Lemma \ref{lem:KSF_from_f(C,G)} gives $$\ol{X}_{G-S'}=\sum_{i'=0}^{|C|-i}\left(\sum_{S''\se C-S',|S'|=i'}f(C-S'-S'',G-S'-S'')\right)\odot \left(\sum_{j=i'}^{|C|-i}\binom{|C|-i-i'}{j-i'}\m_{1^j}\right).$$ Using the binomial theorem for $\odot$ multiplication, the right sum becomes $$\m_{1^{i'}}\odot(\m_1+1)^{\odot (|C|-i-i')},$$ where $f^{\odot n}$ denotes the $\odot$ product of $n$ copies of $f$ with itself. Summing over all $S'$ with $|S'|=i$ and using the change of variables $S=S'\sqcup S''$ and $j=i+i'$, we get $$ \sum_{S'\se C,|S'|=i}\ol{X}_{G-S'}=\sum_{j=i}^{|C|}\left(\sum_{S'\se S \se C,|S|=j}f(C-S,G-S)\odot \m_{1^{j-i}}\odot (\m_1+1)^{\odot(|C|-j)}\right).$$ If we group the terms on the right side by $S$, then for each choice of $S$ with $|S|=j$, there are $\binom ji$ choices for $S'$, so the sum becomes $$\sum_{S'\se C,|S'|=i}\ol{X}_{G-S'}=\sum_{j=i}^{|C|}\left(\sum_{S \se C,|S|=j}f(C-S,G-S)\right)\odot \left(\binom ji\m_{1^{j-i}}\odot (\m_1+1)^{\odot(|C|-j)}\right).$$ The inductive hypothesis implies that for each $i+1<j<|C|$, the sum $$\sum_{S\se C,|S|=j}f(C-S,G-S)$$ can be written in terms of sums of KSFs of the desired form. Rearranging to isolate the $j=i$ terms shows that the sum for $j=i$ can also be written in terms of sums of KSFs that form, since the sum on the left side is also a sum of KSFs of that form. This completes the induction.
\end{proof}

Combining Lemmas \ref{lem:KSF_from_f(C,G)} and \ref{lem:f's_using_X_(G-S)'s} completes the proof of Theorem \ref{cliqueKvertex}.
\end{proof}

We will now give some examples of Theorem \ref{cliqueKvertex} for $|C_1|=|C_2|=2.$  Note that if $|C_i|=2$ and $C_i=\{u_i,v_i\}$ for $i=1,2$, the relevant sums of KSFs are $\ol{X}_{G_i}$ for $j=0$, $\ol{X}_{G_i-u_i}+\ol{X}_{G_i-v_i}$ for $j=1$, and $\ol{X}_{G_i-\{u_i,v_i\}}$ for $k=2$. Assume $\ol{X}_{G_1}=\ol{X}_{G_2}$, $\ol{X}_{G_1-\{u_1,v_1\}}=\ol{X}_{G_2-\{u_2,v_2\}}$, and that $\ol{X}_{G_1'}=\ol{X}_{G_2'}$ for $H=K_1$ a single vertex. Then it automatically follow that $\ol{X}_{G_1-u_1}+\ol{X}_{G_1-v_1}=\ol{X}_{G_2-u_2}+\ol{X}_{G_2-v_2}$, because Lemmas \ref{lem:KSF_from_f(C,G)} and \ref{lem:f's_using_X_(G-S)'s} let us write $\ol{X}_{G_i'}$ in terms of $\ol{X}_{G_i}$, $\ol{X}_{G_i-u_i}+\ol{X}_{G_i-v_i}$, and $\ol{X}_{G_i-\{u_i,v_i\}}$ in a way that depends on all three of these terms. Thus, we can isolate $\ol{X}_{G_i-u_i}+\ol{X}_{G_i-v_i}$ in these sums to show that $\ol{X}_{G_1-u_1}+\ol{X}_{G_1-v_1}=\ol{X}_{G_2-u_2}+\ol{X}_{G_2-v_2}$ assuming all the other terms. 

\begin{example}
    For our graphs from the start of this section, this lets us apply Theorem \ref{cliqueKvertex} to the following combinations of vertex pairs $((u_1,v_1),(u_2,v_2))$:
    \begin{itemize}
        \item \tb{1st graph pair:} None.
        \item \tb{2nd graph pair:} $((1,3),(4,8)),$ $((1,3),(7,8))$.
        \item \tb{3rd graph pair:} $((1,5),(1,5)),$ $((3,8),(3,8)).$
        \item \tb{4th graph pair:} None.
    \end{itemize}
\end{example}

Next, we give another similar example of a way to use one of the above graph pairs to construct more graph pairs with equal KSF:

\begin{prop}\label{prop:attach_to_one_vertex}
    Consider graphs $G_1$ and $G_2$ from the 2nd pair with equal KSFs. Take any graphs $H_1$ and $H_2$ with $\ol{X}_{H_1} = \ol{X}_{H_2}$, and build $G_1'$ by connecting all vertices of $H_1$ to only vertex 5 in $G_1$, and $G_2'$ by connecting all vertices of $H_2$ to only vertex 5 in $G_2.$ Then $\ol{X}_{G_1'} = \ol{X}_{G_2'}.$
\end{prop}
\begin{proof}
    We will again show that the graphs have the same $\m$-expansion. In $G_1,$ the stable sets involving $5$ are $\{5\}$, $\{1,5\}$, and $\{3,5\}$ and in $G_2,$ they are $\{5\},$ $\{4,5\}$, and $\{5,8\}.$ For each of $i=1,2,$ we can split the SSCs of $G_i'$ into cases based on which non-neighbors of 5 are covered \emph{only} by a stable set involving 5 and not by any other stable sets. Write $u_i$ and $v_i$ for the two non-neighbors of 5 in $G_i$, so we can take $u_1 = 1$, $v_1 = 3,$ $u_2 = 4,$ and $v_2 = 8.$
    \begin{itemize}
        \item \tb{Case 1:} $u_i$ is only covered by $\{5,u_i\}$, but $v_i$ is covered by some stable set besides $\{5,v_i\}$. In this case, an SSC of $G_i'$ can be broken into the following parts:
        \begin{enumerate}
            \item Some SSC of $(G_i-\{5,u_i\})\sqcup H_i$, where $G_i-\{5,u_i\}$ is the graph formed from $G_i$ by deleting the two vertices 5 and $u_i.$ We take a disjoint union with $H_i$ since $H_i$ has no neighbors in $G_i-\{5,u_i\}$, as its only neighbor in $G_i$ was 5. The sum of the $\m$-terms associated to these stable set covers is thus $\ol{X}_{(G_i-\{5,u_i\})\sqcup H_i} = \ol{X}_{G_i-\{5,u_i\}} \ol{X}_{H_i}$ (where we are using normal multiplication rather than $\odot$ multiplication).
            \item The required set $\{5,u_i\}$ of size 2, which gives an $\m_2.$
            \item Possibly the stable set $\{5\}$, which is optional since 5 is already covered by $\{5,u_i\}$. We get an $\m_1+1$ for the options of including $\{5\}$ or not.
            \item Possibly the stable set $\{5,v_i\}$, which is also optional since both 5 and $v_i$ are already covered. We get an $\m_2+1$ for the options of including $\{5,v_i\}$ or not.
        \end{enumerate}
        A full SSC is the disjoint union of the partial SSCs for each of these 4 parts, so we can get the generating series for the possible SSCs by taking the $\odot$ product of the generating series for the parts, giving $$(\ol{X}_{G_i - \{5,u_i\}}\ol{X}_{H_i})\odot \m_2 \odot (\m_{1} + 1) \odot (\m_{2} + 1).$$
        
        \item \tb{Case 2:} $u_i$ is covered by a stable set besides $\{5,u_i\}$, but $v_i$ is covered only by the stable set $\{5,v_i\}$. By identical reasoning to Case 1, the generating series for SSCs in this case is $$(\ol{X}_{G_i - \{5,v_i\}}\ol{X}_{H_i}) \odot \m_2 \odot (\m_{1} + 1) \odot (\m_{2} + 1).$$

        \item \tb{Case 3:} Both $u_i$ and $v_i$ are covered by stable sets not involving 5 (but optionally also by $\{5,u_i\}$ and $\{5,c_i\}$). An SSC of $G_i'$ in this case can be broken into two parts:
        \begin{enumerate}
            \item Some SSC of $(G_i-5)\sqcup H_i,$ since all vertices of $G_i-5$ must be covered by a stable set not involving 5. The generating series for these SSCs is $\ol{X}_{(G_i-5)\sqcup H_i} = \ol{X}_{G_i-5}\ol{X}_{H_i}.$
            \item At least one of $\{5\}$, $\{5,u_i\}$, or $\{5,v_i\}$, to ensure that 5 is covered. We can essentially choose independently whether or not to use each of these three stable sets, giving $(\m_1+1)\odot(\m_2+1)\odot(\m_2+1)$, where the $\m_1+1$ represents the option of using $\{5\}$ or not, and the two $\m_2+1$ factors represent the options of using $\{5,u_i\}$ or not, and of using $\{5,v_i\}$ or not. However, we need to subtract 1, since that represents the case where none of the three sets is used.
        \end{enumerate}
        Putting this together, the generating series for SSCs in this case is $$(\ol{X}_{G_i-5}\ol{X}_{H_i})\odot ((\m_{1} + 1) \odot (\m_{2} + 1) \odot (\m_{2} + 1) - 1).$$
        
        \item \tb{Case 4:} Both $u_i$ and $v_i$ are covered only by the stable sets $\{5,u_i\}$ and $\{5,v_i\}$ involving 5. An SSC of $G_i'$ in this case can be broken into the following parts:
        \begin{enumerate}
            \item An SSC of $(G_i-\{5,u_i,v_i\})\sqcup H_i.$ The generating series for these is $\ol{X}_{G_i-\{5,u_i,v_i\}}\ol{X}_{H_i}.$
            \item The required stable set $\{5,u_i\}$ covering $u_i$, giving an $\m_2.$
            \item The required stable set $\{5,v_i\}$ covering $v_i$, giving another $\m_2.$
            \item Optionally the stable set $\{5\}$, giving an $\m_1+1.$
        \end{enumerate}
        Putting this together, we get $$(\ol{X}_{G_i-\{5,u_i,v_i\}}\ol{X}_{H_i})\odot \m_2\odot \m_2 \odot (\m_1+1).$$ as the generating series for SSCs in this case.
    \end{itemize}
    To prove Proposition \ref{prop:attach_to_one_vertex}, we must now check that summing these four cases to get $\ol{X}_{G_i'}$ gives the same result for $i=1,2.$ Note first that if we delete $5$ from both $G_1$ and $G_2$, the resulting graphs are isomorphic, so $\ol{X}_{G_1-5}=\ol{X}_{G_1-5}.$ This implies that the generating series for Case 3 is the same for both graphs, since we also assumed that $\ol{X}_{H_1}=\ol{X}_{H_2}$. Similarly, we can also verify that the graphs formed by deleting $5,u_i,$ and $v_i$ from $G_i$ are isomorphic for $i=1,2$, so $\ol{X}_{G_1-\{5,u_1,v_1\}} = \ol{X}_{G_2-\{5,u_2,v_2\}},$ and thus the generating series for Case 4 is the same for both graphs. Thus, it suffices to show that the sum of the generating series from Cases 1 and 2 is also the same. Since both $\odot$ multiplication and normal multiplication distribute over addition, we can factor the sum as $$((\ol{X}_{G_i-\{5,u_i\}}+\ol{X}_{G_i-\{5,v_i\}})\ol{X}_{H_i}) \odot \m_2 \odot (\m_1 + 1) \odot (\m_2 + 1).$$ Thus, it suffices to have $$\ol{X}_{G_1-\{5,u_1\}} + \ol{X}_{G_1-\{5,v_1\}} = \ol{X}_{G_2-\{5,u_2\}} + \ol{X}_{G_2-\{5,v_2\}}.$$ In this case, none of the graphs $G_i-\{5,u_i\}$ and $G_i-\{5,v_i\}$ are isomorphic to each other, but we can still show that their sum must be the same. Setting $H_i=\varnothing,$ it follows from the fact that $\ol{X}_{G_1}=\ol{X}_{G_2}$ (which we checked with our code) together with the fact that the Case 3 and Case 4 parts match up, that $(\ol{X}_{G_i-\{5,u_i\}}+\ol{X}_{G_i-\{5,v_i\}})\odot \m_2 \odot(\m_1 + 1)\odot(\m_2 + 1)$ must be the same for both graphs. Then it follows from the fact that $\ol{\widetilde{\Lambda}}$ is an integral domain that the sum $\ol{X}_{G_i-\{5,u_i\}}+\ol{X}_{G_i-\{5,v_i\}}$ must also be the same for $i=1,2,$ completing the proof of Proposition \ref{prop:attach_to_one_vertex}.
\end{proof}
We can generalize Proposition~\ref{prop:attach_to_one_vertex} to give another sufficient condition for building graphs with equal KSF by connecting new graphs to all but one vertex of the existing graphs:
\begin{prop}\label{prop:OneVertex_General}
    Let $G_1, G_2$ be graphs such that $\overline{X}_{G_1} = \overline{X}_{G_2}$, and let $H_1, H_2$ be graphs such that $\overline{X}_{H_1} = \overline{X}_{H_2}$. Consider a vertex $v_{1} \in V(G_1)$ such that the subgraph of $G_1$ induced by all vertices not adjacent to $v_1$ forms a clique $C_1$ of size $k$. Similarly, let $v_2 \in V(G_2)$ be a vertex such that the subgraph of $G_2$ induced by all vertices not adjacent to $v_2$ forms a clique $C_2$ of the same size $k$. Let $G_1'$ be the graph obtained by connecting every vertex of $H_1$ to the vertex $v_1$ in $G_1$ (and to no other vertex of $G_1$). Similarly, let $G_2'$ be the graph obtained by connecting every vertex of $H_2$ to the vertex $v_2$ in $G_2$ (and to no other vertex of $G_2$). If  for every $0\le j\le k$,
    $$\sum_{|S|\subseteq C_1, |S| = j} \overline{X}_{G_{1} - (\{v_{1}\} \cup S)} = \sum_{S\subseteq C_2, |S| = j}\overline{X}_{G_{2} - (\{v_{2}\} \cup S)},$$ then $\overline{X}_{G_1'} = \overline{X}_{G_2'}$.
\end{prop}

\begin{proof}
    We will show that the graphs have the same $\m$ expansion. Let the vertices not adjacent to $v_1$ in $G_1$ be $v_{1, 1}, v_{1, 2}, \ldots, v_{1, k}$. We will use casework on how many of $v_{1,1},\dots, v_{1,k}$ are covered only by stable sets involving $v_1.$ Without loss of generality, we may assume that $v_{1, 1}, v_{1, 2}, \ldots, v_{1, j}$ are covered only by stable sets involving $v_1$, while $v_{1, j + 1}, v_{1, j + 2}, \ldots, v_{1, k}$ are covered by other stable sets, not just stable sets involving $v_1$. We first consider the case where $j \neq 0$; the case $j = 0$ requires separate treatment.

    To form a stable set cover of $G_1'$, we must include the stable sets $\{v_{1}, v_{1, 1}\}, \{v_{1}, v_{1, 2}\}, \ldots,$ $\{v_1, v_{1, j}\}$. This contributes $\m_{2^{j}}$ to the expansion. Additionally, we need to cover the remaining vertices $$H_1 \sqcup (G_1 - \{v_{1}, v_{1, 1}, v_{1, 2}, \ldots, v_{1, k}\}),$$ and the generating series for ways to do this is $\overline{X}_{H_1} \overline{X}_{G_1 - \{v_{1}, v_{1, 1}, v_{1, 2}, \ldots, v_{1, j}\}}$, because we know no vertices in $H$ are adjacent to any vertex except $v_1,$ and KSFs multiply over disjoint unions. We may optionally also include some of the stable sets $\{v_{1}\}, \{v_1,v_{1, j + 1}\}, \ldots, \{v_1,v_{1, k}\}$, which contributes $$(\m_{1} + 1) \odot \underbrace{(\m_{2} + 1) \odot \ldots \odot (\m_{2} + 1)}_{k - j \text{ times}} = (\m_{1} + 1) \odot (\m_{2} + 1)^{\odot (k - j)},$$ where the notation $f^{\odot n}$ means the $\odot$ product of $f$ with itself $n$ times, and each term represents the choice to include one of those stable sets or not.
    
    Combining these contributions, for $j \neq 0$ we obtain:
    $$\m_{2^{j}} \odot (\overline{X}_{H_{1}} \overline{X}_{G_{1} - \{v_{1}, v_{1, 1}, v_{1, 2}, \ldots, v_{1, k}\}}) \odot (\m_{1} + 1) \odot (\m_{2} + 1)^{\odot (k - j)}.$$
    
    For the case $j = 0$, we must ensure that some stable set covers $v_{1}$, so we subtract the case where this does not happen, meaning none of the sets $\{v_1\},$ $\{v_1,v_{1,j+1}\},\dots,\{v_1,v_{1,k}\}$ get used, i.e. we choose a 1 instead of an $\m_1$ or $\m_2$ from every term. Thus, the $j=0$ term instead becomes
    $$(\overline{X}_{H_{1}} \overline{X}_{G_{1} - \{v_{1}\}}) \odot \left((\m_{1} + 1) \odot (\m_{2} + 1)^{\odot k} - 1\right).$$

    We now sum both expressions over all subsets of $C_1=\{v_{1, 1}, v_{1, 2}, \ldots, v_{1, k}\}$. For a fixed value of $j$, we sum over all subsets of size $j$. Note that for fixed $j$, the only term that varies is $\overline{X}_{G_{1} - \{v_{1}, v_{1, 1}, v_{1, 2}, \ldots, v_{1, j}\}}$. Thus, for fixed $j \neq 0$, the sum becomes:
    $$(\m_{1} + 1) \odot (\m_{2} + 1)^{\odot (k - j)} \odot \m_{2^{j}} \odot \left(\overline{X}_{H_{1}} \left(\sum_{S\subseteq C_1,|S| = j} \overline{X}_{G_{1} - (\{v_{1}\} \cup S)}\right)\right),$$
    where $S$ denotes our subset. For $j = 0$, there is no summation:
    $$\left((\m_{1} + 1) \odot (\m_{2} + 1)^{\odot k} - 1\right) \odot (\overline{X}_{H_{1}} \overline{X}_{G_{1} - \{v_{1}\}}).$$

    By symmetry, the analysis for $G_2$ and $H_2$ yields identical expressions with the variables appropriately swapped. Since $\overline{X}_{H_1} = \overline{X}_{H_2}$ and $\overline{X}_{G_{1}} = \overline{X}_{G_2}$, we require only that the summations be equal:
    $$\sum_{|S| = j} \overline{X}_{G_{1} - (\{v_{1}\} \cup S)} = \sum_{|S| = j} \overline{X}_{G_{2} - (\{v_{2}\} \cup S)}.$$
    for all $j$. This is precisely the condition given in the proposition statement, and therefore $\overline{X}_{G_{1}'} = \overline{X}_{G_{2}'}$.
\end{proof}
The above sorts of methods of combining known graph pairs with equal KSF can likely be generalized to construct many similar infinite families of nonisomorphic graph pairs with equal KSF. However, another potential approach for finding graphs with equal KSF is to make use of the \emph{\tb{\tcb{contraction-deletion relation}}} from \cite{crew2023kromatic}, as it is a modified version of the contraction-deletion relation from \cite{crew2020deletion} for the vertex-weighted CSF, which was used in \cite{aliste2021vertex} to construct graph pairs with equal CSF. Since the contraction-deletion relation for the KSF is more complicated, it may be impossible to find nonisomorphic graphs that can be shown to have the same KSF with a single contraction-deletion step (like the authors of \cite{aliste2021vertex} did for the CSF). However, it is possible one could use some short sequence of contraction-deletion steps to show that certain graphs have the same KSF. This suggests the following question:

\begin{question}
    Can the contraction-deletion relation from \cite{crew2023kromatic} be used to explain why our graph pairs have equal KSF, or to construct more graph pairs with equal KSF?
\end{question}

\section{Orellana-Scott construction} \label{sec:orellana-scott}

In the remainder of this paper (this section and \S\ref{sec:split_graphs}), we look at two different constructions for graph pairs with equal CSF, and we try to check whether they can also have equal KSF. For both constructions, we have not found any examples where the graphs do have the same KSF, but we can instead prove that the graphs have different KSFs under certain assumptions despite having the same CSF, thus giving new examples of cases where the KSF is a more powerful invariant than the CSF.

Orellana and Scott \cite{orellana2014graphs} found the following construction that results in an infinite number of pairs of nonisomorphic graphs with the same CSF:
\begin{theorem}[Orellana-Scott~\cite{orellana2014graphs}, Theorem 4.2]\label{thm:orellana-scott}
    Let $G = (V,E)$ be a graph that has four vertices $u$, $v$, $w$, $z$ with the property that $uz, wz, vw \in E$ and $uw, vz, uw \notin E$. If there exists a graph automorphism $\varphi : G-wz \to G-wz$ such that
    \begin{align*}
    \varphi(\{u, w\}) = \{v, z\} \hspace{1cm}\text{ and } \hspace{1cm}\varphi(\{v, z\}) = \{u, w\}
    \end{align*}
    then
    \begin{align*}
    H := G\cup uw \hspace{1cm}\text{ and } \hspace{1cm}J := G\cup vz
    \end{align*}
    have the same chromatic symmetric function.
\end{theorem}

\begin{example}
    One example of the construction in Theorem \ref{thm:orellana-scott} is shown below:
    \begin{center}
        \includegraphics[width=6cm]{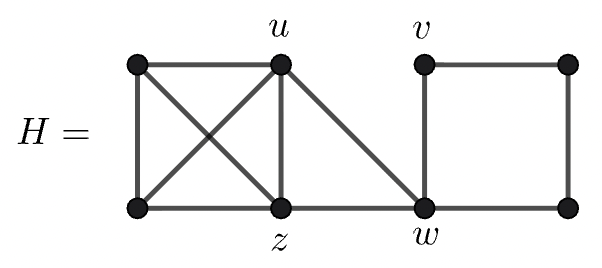}
        \hspace{1cm}
        \includegraphics[width=6cm]{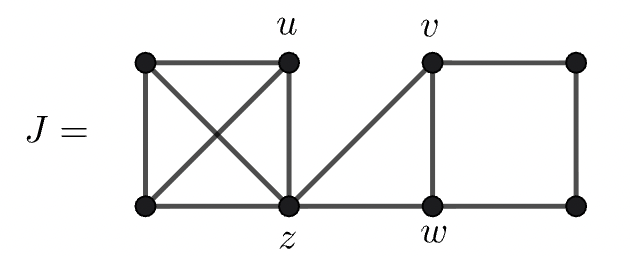}
    \end{center}
\end{example}

Orellana and Scott also observed that Theorem \ref{thm:orellana-scott} can be used to construct an infinite family of pairs of unicyclic graphs with the same CSF, by attaching a tree $T_1$ to each of $u$ and $z$, and a different tree $T_2$ to each of $v$ and $w$.

We show now that subject to certain conditions, the KSF does distinguish the pairs of graphs from Theorem \ref{thm:orellana-scott} even though the CSF does not:

\begin{theorem}\label{thm:orellana-scott-KSF}
    Assume $\varphi(u)=z,$ $\varphi(z)=u$, $\varphi(w) = v$, and $\varphi(v) = w$. Then $\ol{X}_H\ne\ol{X}_J$ as long as the sum of the number of neighbors of $w$ adjacent to neither $u$ nor $v$, plus the number of neighbors of $w$ adjacent to neither $z$ nor $v$, is different from the number of neighbors of $z$ adjacent to neither $u$ nor $v$, plus the number of neighbors of $z$ adjacent to neither $u$ nor $w$.
\end{theorem}

Note in particular that in the unicyclic case where we attach a tree $T_1$ to each of $u$ and $z$ and a different tree $T_2$ to each of $v$ and $w$, the KSFs are different as long as the number of edges attaching $T_2$ to $w$ is different from the number of edges attaching $T_1$ to $z$.

\begin{proof}
    To prove this, we count the number of induced subgraphs in $H$ and $J$ that are isomorphic to the claw graph, shown below: 
    \begin{center}   
    \includegraphics[scale=0.2]{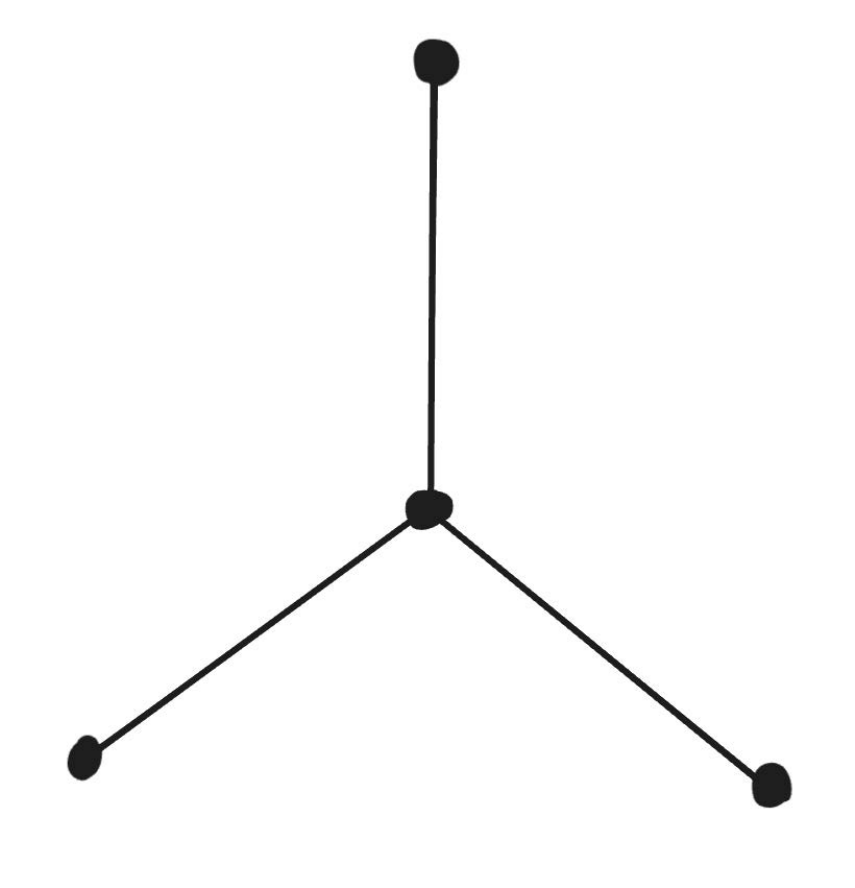}
\end{center}
    \noindent Since the claw is one of the independence unique graphs from Theorem \ref{thm:graph_list}, if we can show that the number of induced claws is different in $H$ than in $J$, it will follow that $\ol{X}_H\ne\ol{X}_J$.
    
    Note that if an induced claw does not contain either both $u$ and $w$ or both $v$ and $z$, then its set of 4 vertices forms an induced claw in both $H$ and $J$, since the exact same set of edges connect those vertices in $H$ as in $J$. So, we can ignore those claws since the number of them is the same in $H$ and $J$. 
    
    Thus, consider induced claws containing both $u$ and $w$ or both $v$ and $z$. We have the following four ways to get induced claws in $H$ containing both $u$ and $w$ or both $v$ and $z$:
    \begin{center}
    \includegraphics[width=13cm]{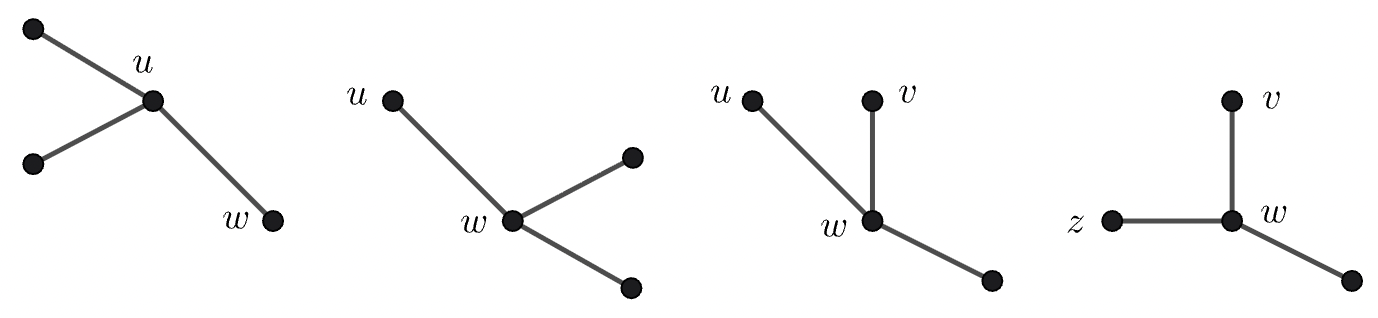}
    \end{center}
    Similarly, we have the following four cases for induced claws in $J$ containing both $u$ and $w$ or both $v$ and $z$:
    \begin{center}
    \includegraphics[width=13cm]{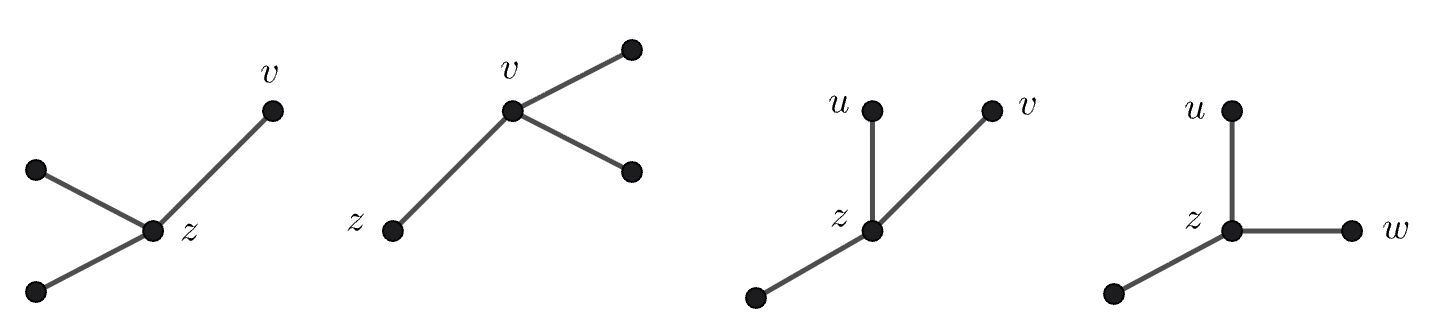}
    \end{center}
    For each claw from the first case for either $H$ or $J$, we can get a bijection from induced claws in $H$ to induced claws in $J$ by applying $\varphi$ to all the vertices, as shown below:
    \begin{center}
        \includegraphics[width=10cm]{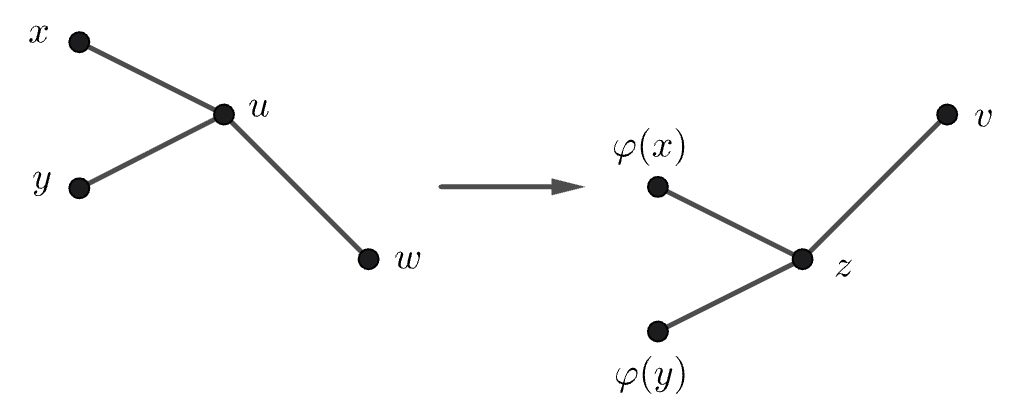}
    \end{center}
    The same thing holds for the second case above:
    \begin{center}
        \includegraphics[width=10.5cm]{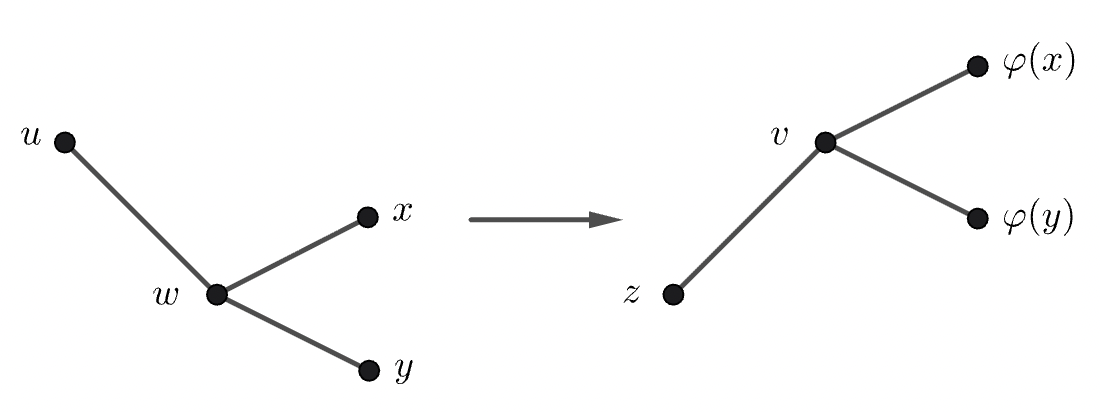}
    \end{center}
    Now in the third case, the number of induced claws in $H$ is equal to the number of neighbors of $w$ adjacent to neither $u$ or $v$, while in $J$ it is equal to the number of neighbors of $z$ adjacent to neither $u$ nor $v.$ Similarly, in the fourth case, the number of induced claws in $H$ is equal to the number of neighbors of $w$ adjacent to neither $z$ nor $v,$ while in $J$, it is equal to the number of neighbors of $z$ adjacent to neither $u$ nor $w$. It follows from the assumption in Theorem \ref{thm:orellana-scott-KSF} that the total number of induced claws is different in $H$ than in $J$, hence $\ol{X}_H\ne \ol{X}_J.$
    \end{proof}

\section{Aliste-Prieto--Crew--Spirkl--Zamora construction}\label{sec:split_graphs}

Aliste-Prieto, Crew, Spirkl, and Zamora \cite{aliste2021vertex} gave a different construction of infinitely many pairs of nonisomorphic graphs with the same CSF. To describe their construction, we first define the split graph $\tn{sp}(G)$. There is a natural way noted by \cite{loebl2019isomorphism} to associate to any possibly non-simple (unweighted) graph a corresponding simple split graph. Let $G$ be a graph with vertex set $V(G) =\{v_{1},...,v_{n}\}$ and edge set $E(G) =\{e_{1},...,e_{m}\}$.  Then the \emph{\tb{\tcb{split graph}}} $\tn{sp}(G)$ corresponding to $G$ has vertex set $$V(\tn{sp}(G)) := V(G)\cup E(G)$$ and edge set $$E(\tn{sp}(G)) :=\{v_{i}v_{j}: 1 \leq i < j \leq n\} \cup \{v_{i}e_{k}, \ v_j e_k:e_{k}=v_{i}v_{j} \hspace{0.1cm} \text{in} \hspace{0.1cm} G\}.$$  In other words, $\tn{sp}(G)$ is formed by taking the vertices of $G$, making them into a clique, and then adding a “hat” corresponding to each edge of $G$.  Using the above notation, we say that vertex $e_k$ of $\tn{sp}(G)$ is the splitting vertex of the edge $e_k=v_{i}v_{j}$ in $G$. An example of the split graph construction is shown below:
\begin{center}
    \includegraphics[width=10cm]{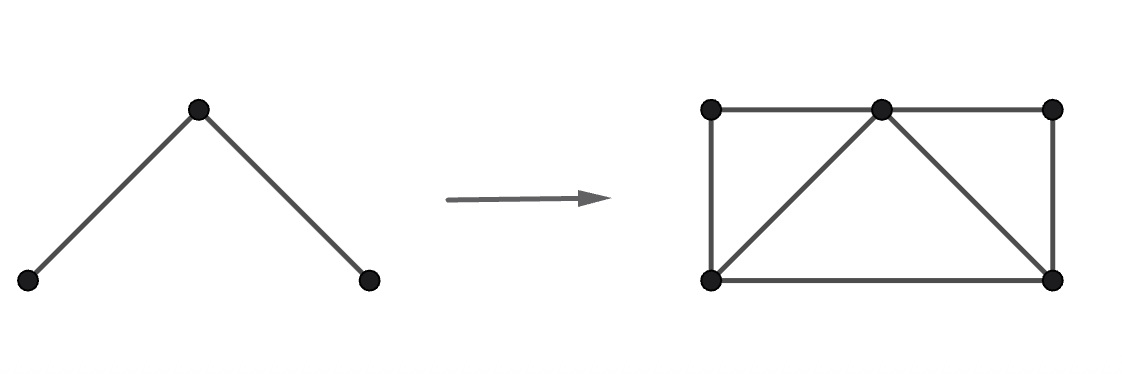}
\end{center}

\begin{theorem}[Aliste-Prieto-Crew-Spirkl-Zamora \cite{aliste2021vertex}, Lemma 8]\label{thm:ALI21}
    Let $G$ be  an  unweighted  graph.   Suppose $G$ has (not  necessarily  distinct) vertices $u,u',v,v'$ such that:
    \begin{itemize}
        \item $uv \not \in E(G)$ \hspace{0.1cm}\text{and}\hspace{0.1cm} $u'v' \not \in E(G)$
        \item There is some automorphism of G that maps $u$ to $u'$, and some (possibly different) automorphism of G that maps $v$ to $v'$.
    \end{itemize}
    Then $X_{\tn{sp}(G \cup uv)} = X_{\tn{sp}(G \cup u'v')}$.
\end{theorem}

\bigskip

Like for the Orellana-Scott construction, we prove that subject to certain conditions, the graph pairs from the construction in \cite{aliste2021vertex} have different KSF even though they have the same CSF:

\begin{theorem}\label{thm:split_graphs}
    Suppose $u$ and $v$ have at least one common neighbor in $G$ but $u'$ and $v'$ have no common neighbors, and suppose also that $|V|\ge 6$, and that at least one of the following conditions holds:
    \begin{enumerate}
        \item One of the common neighbors of $u$ and $v$ in $G$ has degree at least 3.
        \item $\deg_G(u)\ge 2$.
        \item $\deg_G(v)\ge 2.$
    \end{enumerate}
    Then $\ol{X}_{\tn{sp}(G\cup uv)}\ne \ol{X}_{\tn{sp}(G\cup u'v')}.$
\end{theorem}

\begin{proof}
    It can be checked that the two graphs $H_1$ and $H_2$ shown below have the same independence polynomial $2x^5 + 9x^4 + 16x^3 + 15x^2 + 7x + 1,$ and that no other graphs have that independence polynomial:
    $$\includegraphics[width=7cm]{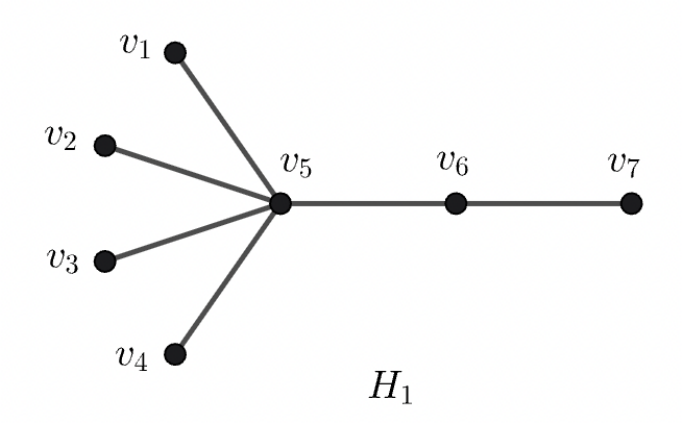}\hspace{1cm}\includegraphics[width=6.5cm]{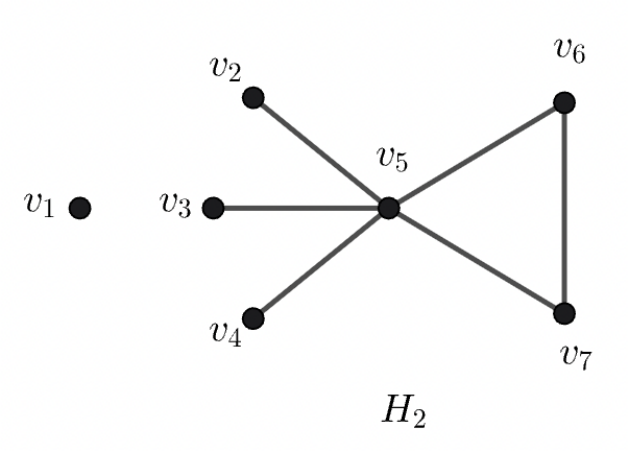}$$
    Thus, it follows from Theorem \ref{thm:ind_poly} that the sum of the number of induced copies of $H_1$ in $\tn{sp}(G\cup uv)$ plus the number of induced copies of $H_2$ in $\tn{sp}(G\cup uv)$ can be determined from $\ol{X}_{\tn{sp}(G\cup uv)},$ and likewise for $\tn{sp}(G\cup u'v')$. In particular, if we can show that this sum is different for $G\cup uv$ than for $G\cup u'v'$, it will follow that $\ol{X}_{\tn{sp}(G\cup uv)}\ne \ol{X}_{\tn{sp}(G\cup u'v')}.$ We will show in particular that this sum is greater in $\tn{sp}(G\cup uv)$ than in $\tn{sp}(G\cup u'v')$.

    \begin{lemma}\label{lem:H_1}
        The number of induced copies of $H_1$ is the same in $\tn{sp}(G\cup uv)$ and $\tn{sp}(G\cup u'v')$.
    \end{lemma}

    \begin{proof}
        In an induced copy of $H_1$ in $\tn{sp}(G\cup uv)$ or $\tn{sp}(G\cup u'v')$, each vertex must come from either a vertex or an edge of the original graph $G\cup uv$ or $G\cup u'v'$. Each edge vertex has degree 2, and $v_5$ of $H_1$ has degree 5, so $v_5$ must come from a vertex of $H_1$. Then if $v_6$ came from an edge of $H_1$, $v_7$ would have to come from a vertex since no two edge vertices are connected to each other. But then $v_5$ and $v_7$ would need to be connected, since they would both come from vertices in the original graph. This is a contradiction, so $v_6$ must come from a vertex of the original graph. Then $v_1,v_2,v_3,$ and $v_4$ must come from edges, since they are not connected to $v_6$, and $v_7$ must also come from an edge, since it is not connected to $v_5$. Thus, the only possibility is that the vertices of $H_1$ labeled ``vertex" below come from vertices of $G\cup uv$ or $G\cup u'v'$ while the ones labeled ``edge" come from edges:
        \begin{center}              \includegraphics[width=0.5\linewidth]{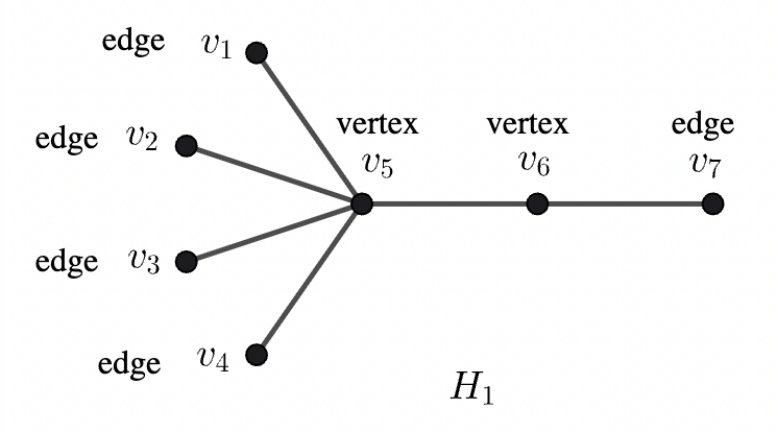}
        \end{center}
        Next, note that if none of the edge vertices are $uv$ or $u'v'$, the induced copy of $H_1$ exists in both split graphs. Thus, it suffices to show that the number of induced copies of $H_1$ in $\tn{sp}(G\cup uv)$ with $uv$ as a vertex equals the number of induced copies of $H_1$ in $\tn{sp}(G\cup u'v')$ with $u'v'$ as a vertex. There are two cases for where this vertex can be:\\
        \\
             \textbf{Case 1.} $uv$ or $u'v'$ is one of the four edge vertices on the left.\\
            \\
            \indent We will give a 1-to-1 correspondence between induced copies of $H_1$ in $\tn{sp}(G\cup uv)$ where $uv$ is one of the left vertices and induced copies of $H_1$ in $\tn{sp}(G\cup u'v')$ where $u'v'$ is one of the left vertices. If we assume $v_1=uv$ in a copy of $H_1$ in $\tn{sp}(G\cup uv)$, then $v_5$ must be either $u$ or $v$. The two cases are identical, so assume $v_5=u$. Then $v_2, v_3,$ and $v_4$ come from three edges $uw,ux,$ and $uy$ of $G$ that are incident to $u$, $v_6$ comes from some other vertex $z\ne u,v,w,x,y$, and $v_7$ comes from an edge $zt$ incident to $z$.
            
            Let $\varphi_u$ be the automorphism of $G$ taking $u$ to $u'$. Then we can generally turn an induced copy of $H_1$ in $\tn{sp}(G\cup uv)$ with $v_1=uv$ and $v_5=u$ into one in $\tn{sp}(G\cup u'v')$ with $v_1=u'v'$ and $v_5=u'$ by simply applying $\varphi_u$ to all the vertices and edges, except that we let $uv$ map to $u'v'$ instead of to $u'\varphi_u(v)$, as shown below: 
            $$\vcenter{\includegraphics[width=6.5cm]{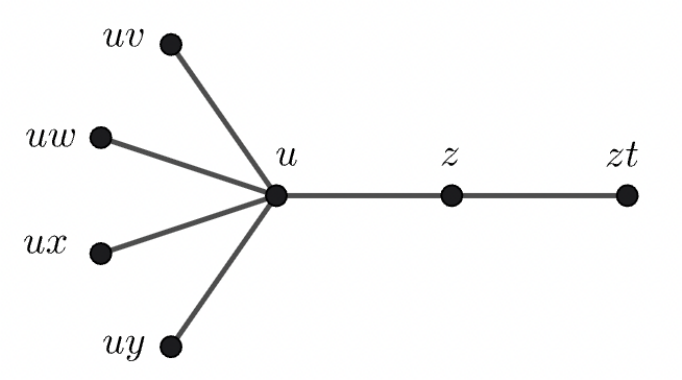} \hspace{1cm}\includegraphics[width=8cm]{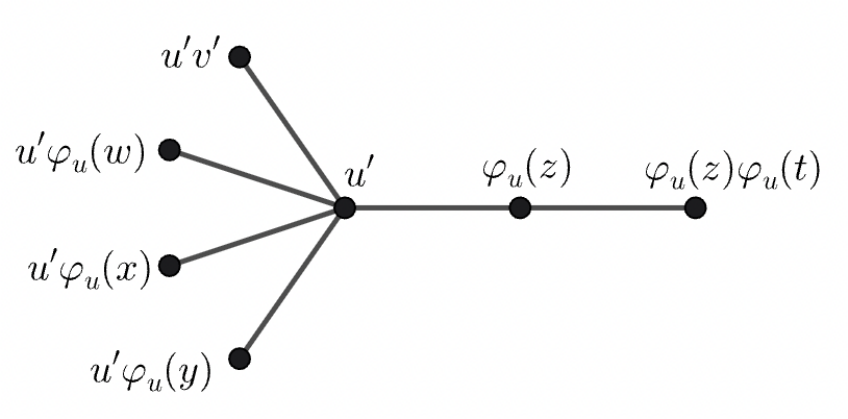}}$$
            The exception where this does not work is if $\varphi_u(z) = v'$, since then vertices $\varphi_u(z)$ and $u'v'$ on the right would be connected. We can assume $\varphi_u(v)\ne v'$, since in that case $G\cup uv$ and $G\cup u'v'$ would be isomorphic via $\varphi$. Then we can instead map $z=\varphi_u^{-1}(v')$ to $\varphi_u(v)$, and the edge $\varphi_u^{-1}(v')s$ incident to $\varphi_u^{-1}(v')$ just needs to map to some edge $\varphi_u(v)t$ incident to $\varphi_u(v)$, as shown below:
            $$\vcenter{\includegraphics[width=7cm]{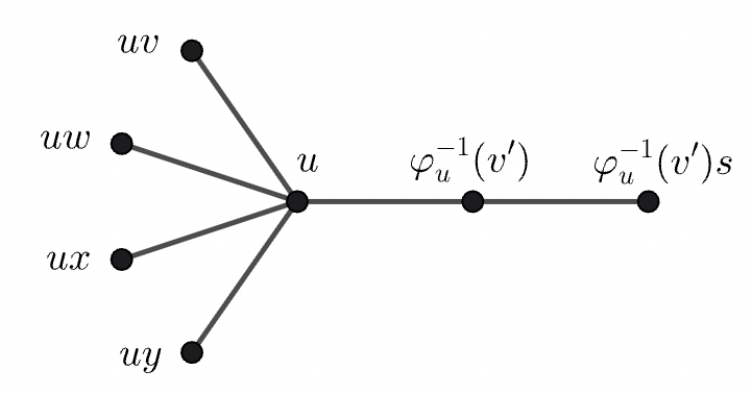} \hspace{1cm}\includegraphics[width=8cm]{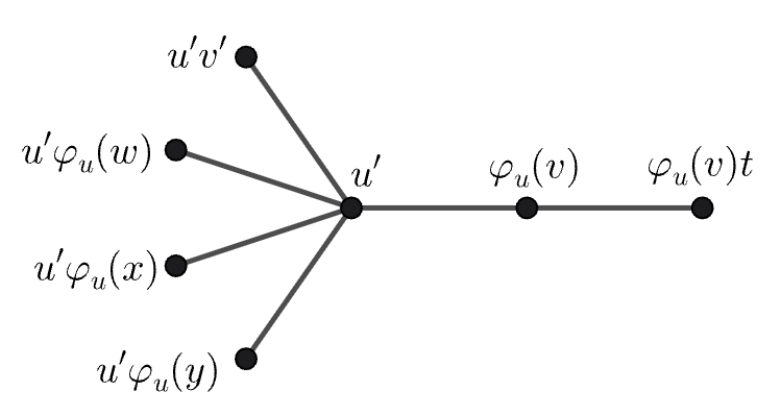}}$$
            The number of such graphs on the left is the same as the number on the right because there is an automorphism taking $v$ to $v'$, hence $v$ and $v'$ have the same degree, and thus so do $\varphi_u(v)$ and $\varphi_u^{-1}(v')$.
            
            The only potential issue is that on the left, the edge incident to $\varphi_u^{-1}(v')$ must not be the edge $u\varphi_u^{-1}(v')$, and similarly, on the right, the edge incident to $\varphi_u(v)$ must not be $u'\varphi_u(v)$. But both those requirements hold automatically, because $uv$ is not an edge in $G$, so applying the automorphism $\varphi_u$ to its endpoints shows that $\varphi_u(u)\varphi_u(v) = u'\varphi_u(v)$ is also not an edge. Similarly, $u'v'$ is not an edge, so applying the automorphism $\varphi_u^{-1}$ to both its endpoints shows that $\varphi^{-1}(u')\varphi^{-1}(v') = u\varphi^{-1}(v)$ is also not an edge.
            
            An identical argument shows that we still get the same number of induced copies of $H_1$ in both graphs if we let $v_5=v$ on the left and $v_5=v'$ on the right. \\
            \\
            \textbf{Case 2.} $uv$ or $u'v'$ is the rightmost edge vertex $v_7$.\\
            \\
            \indent The argument in this case is essentially the same as in Case 1. If $v_7=uv$ in an induced copy of $H_1$ in $\tn{sp}(G\cup uv)$, then $v_6\in\{u,v\}$, so without loss of generality assume $v_6 = u$. Then we can generally get a corresponding induced copy of $H_1$ in $\tn{sp}(G\cup u'v')$ with $v_6=u'$ and $v_7=u'v'$ by simply applying $\varphi_u$ to all vertices and edges, except that we send $uv$ to $u'v'$ as before:
            $$\includegraphics[width=6.5cm]{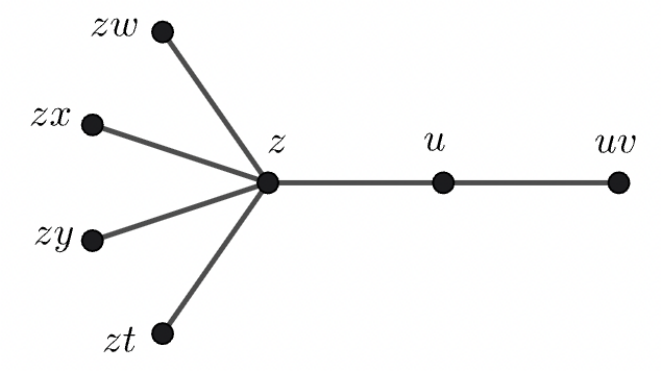}\hspace{1cm}\includegraphics[width=8cm]{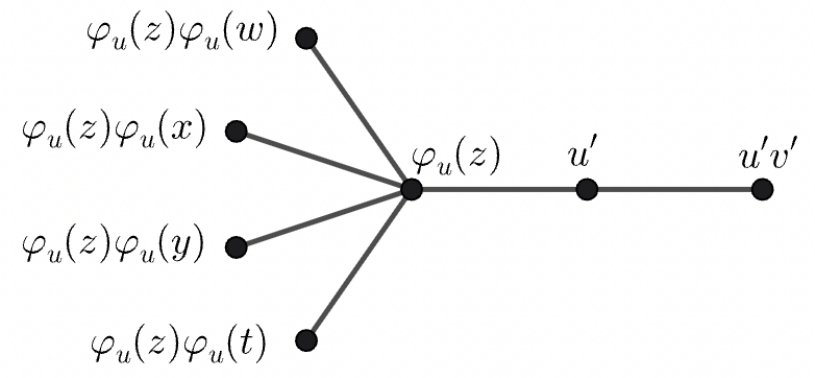}$$
            The exception is when $v_5 = \varphi^{-1}_u(v')$ in $\tn{sp}(G\cup uv)$, in which case we set $v_5 = \varphi_u(v)$ in the corresponding copy of $H_1$ in $\tn{sp}(G\cup u'v')$:
            $$\includegraphics[width=6.5cm]{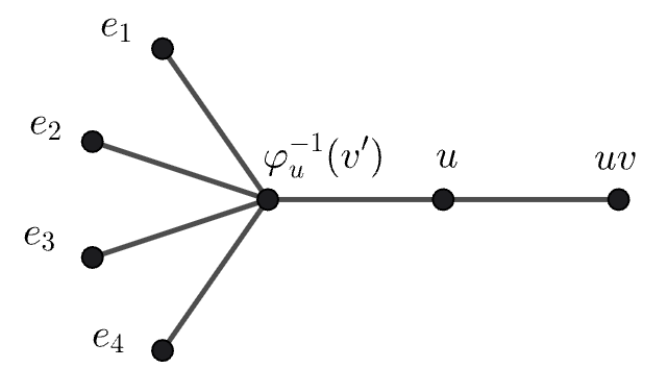}\hspace{1cm}\includegraphics[width=6.5cm]{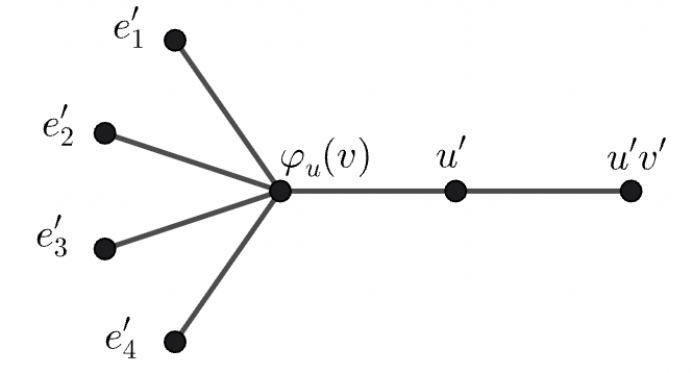}$$
            Then the number of choices for the left vertices $v_1,v_2,v_3,$ and $v_4$ is the same in both cases because $\varphi_u^{-1}(v')$ and $\varphi^{-1}_u(v')$ have the same degree as noted above, so the number of ways to choose four edges $e_1,e_2,e_3,$ and $e_4$ incident to $\varphi_u^{-1}(v')$ is the same as the number of ways to choose four edges $e_1',e_2',e_3',$ and $e_4'$ incident to $\varphi_u(v).$ As in Case 1, we do not have to worry about any of these edges having $u$ or $u'$ as an endpoint, since $u'\varphi_u(v)$ and $u\varphi^{-1}(v')$ are not edges in $G$.

            As in Case 1, the argument is identical if we instead let $v_6 = v$ in $\tn{sp}(G\cup uv)$ and $v_6=v'$ in $\tn{sp}(G\cup u'v')$.
    \end{proof}

    \begin{lemma}\label{lem:H_2}
        The number of induced copies of $H_2$ is greater in $\tn{sp}(G\cup uv)$ than in $\tn{sp}(G\cup u'v')$.
    \end{lemma}

    \begin{proof}
    Again, we first consider which vertices of an induced copy of $H_2$ could come from vertices of $G\cup uv$ or $G\cup u'v'$ and which ones could come from edges. Since $v_5$ has degree 5, it must come from a vertex. Since $v_6$ and $v_7$ are connected to each other, they cannot both come from edges, but they could come from either an edge and a vertex or both vertices.  Then $v_1,v_2,v_3,$ and $v_4$ must all come from edges since they are not connected to $v_6$ or $v_7$, at least one of which comes from a vertex. We thus have two cases: $v_6$ and $v_7$ come from two vertices or an edge and a vertex.\\
    \\
    \textbf{Case 1.} $v_6$ comes from an edge and $v_7$ from a vertex.
    \begin{center}    \includegraphics[width=7cm]{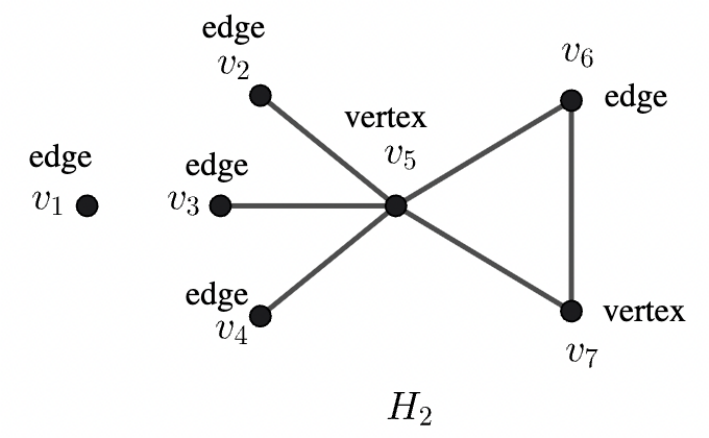}
    \end{center}
    \indent We will show that in this case, the number of induced copies of $H_2$ in the two graphs is the same. As in Case 1, for each induced copy of $H_2$ in $\tn{sp}(G\cup uv)$ where $uv$ is not a vertex, there is a corresponding copy in $\tn{sp}(G\cup u'v')$ where $u'v'$ is not a vertex (by just using all the same edges and vertices), so it suffices to show that there are the same number of copies of $H_2$ in $\tn{sp}(G\cup uv)$ using $uv$ as a vertex as there are copies in $\tn{sp}(G\cup u'v')$ using $u'v'$ as a vertex. We now consider subcases based on the location of $uv$ or $u'v'$.\\
    \\
    \textbf{Subcase 1.1.} $uv$ or $u'v'$ is the lone vertex $v_1$ on the left.\\
    \\
    \indent Then the induced copy of $H_2$ in $\tn{sp}(G\cup uv)$ looks something like this, where $w$ and $x$ are vertices different from $u$ and $v$, and $y,z,$ and $t$ are neighbors of $w$ that potentially could be $u$ or $v$:
    \begin{center}
        \includegraphics[width=6cm]{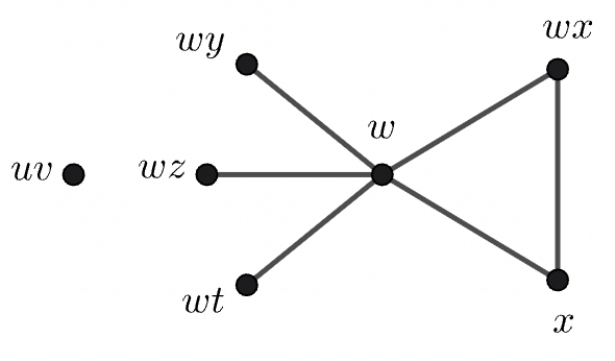}
    \end{center}
    The number of ways to do this is 
    \begin{equation}\label{eqn:subcase1.1}
        \sum_{w\ne u,v}\binom{\deg_G(w)}{4}\cdot 4 - \sum_{w\sim u}\binom{\deg_G(w)-1}{3}-\sum_{w\sim v}\binom{\deg_G(w)-1}{3},
    \end{equation}
    where we use $w\sim u$ to mean that $w$ is adjacent to $u$ in $G$. The first sum comes from choosing $w$, choosing any 4 neighbors of $w$, and then choosing one of those 4 neighbors to be $x$. The subtracted sums come from the fact that we need to subtract cases where $x=u$ or $x=v$, in which case $w$ must be a neighbor of $u$ (or $v$), and $y,z,$ and $t$ can be any other neighbors of $w$. But (\ref{eqn:subcase1.1}) stays the same if we replace $u$ with $u'$ and $v$ with $v'$ since there are automorphisms taking $u$ to $u'$ and $v$ to $v'$, so the number of copies of $H_2$ in this case is the same for $\tn{sp}(G\cup uv)$ as for $\tn{sp}(G\cup u'v')$.\\
    \\
    \noindent\textbf{Subcase 1.2.} $uv$ or $u'v'$ is $v_2$.\\
    \\
    \indent In this case, $v_5$ must be $u$ or $v$ in $\tn{sp}(G\cup uv)$, so without loss of generality assume $v_5 = v$. Then we get the following picture, where $x$ is a neighbor of $u$ and $e$ is an edge not incident to $u$ or $x$:
    \begin{center}
        \includegraphics[width=6cm]{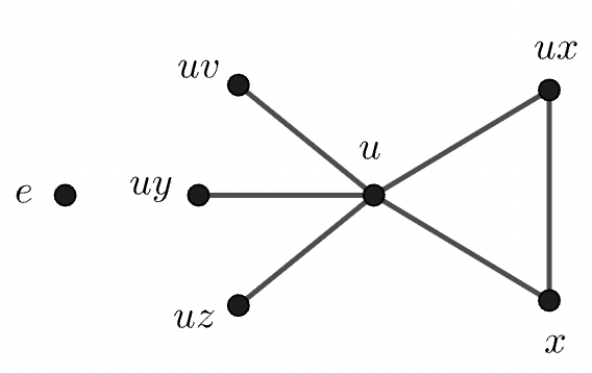}
    \end{center}
    The number of ways to do this is $$\sum_{x\sim u}\binom{\deg_G(u)-1}{2}\cdot(|E|-\deg_G(u)-\deg_G(x)+1),$$ because once $x$ is chosen, there are $\binom{\deg_G(u)-1}2$ ways to choose the other two neighbors $y$ and $z$ of $u$, and then there are $|E|-\deg_G(u)-\deg_G(x)+1$ ways to choose the other edge $e$ because the only restriction is that it cannot be incident to $u$ or $x$ (and the $+1$ comes from inclusion-exclusion, since the edge $ux$ got subtracted twice and must be added back). This expression stays the same if we replace $u$ with $u'$, so we get the same number of copies of $H_2$ in $\tn{sp}(G\cup uv)$ with $v_5=u$ as we do in $\tn{sp}(G\cup u'v')$ with $v_5=u'$, and by analogous reasoning, we also get the same number of copies with $v_5 = v$ or $v'$.\\
    \\
    \textbf{Subcase 1.3.} $uv$ or $u'v'$ is $v_6$.\\
    \\
    \indent If $v_6 = uv$ in $\tn{sp}(G\cup uv)$, then $v_5$ and $v_7$ must be $u$ and $v$ in some order, so without loss of generality assume $v_5 = u$ and $v_7 = v$. Then we need to choose $v_3,v_3,$ and $v_4$ to be three edges incident to $u$, and $v_1$ to be any edge incident to neither $u$ nor $v$, as shown below:
    \begin{center}
        \includegraphics[width=6cm]{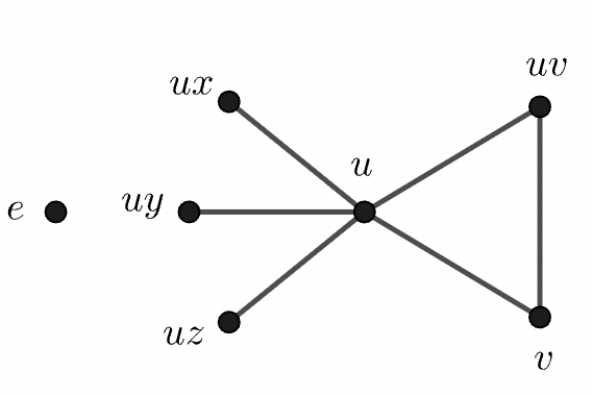}
    \end{center}
    The number of ways to do this is $$\binom{\deg_G(u)}{3}\cdot(|E|-\deg_G(u)-\deg_G(v)).$$ That expression stays the same if we swap $u$ with $u'$ and $v$ with $v'$, and so does the analogous expression with $u$ and $v$ interchanged, so we again get the same number of copies of $H_2$ in $\tn{sp}(G\cup uv)$ as in $\tn{sp}(G\cup u'v')$ in this case.\\
    \\
    \textbf{Case 2.} Both $v_6$ and $v_7$ come from vertices.\\
    \\
    \indent This is the interesting case where we will show that there are more copies of $H_2$ in $\tn{sp}(G\cup uv)$ than in $\tn{sp}(G\cup u'v')$. The figure below shows which vertices of $H_2$ come from vertices and which ones come from edges:
    \begin{center}    
    \includegraphics[width=7cm]{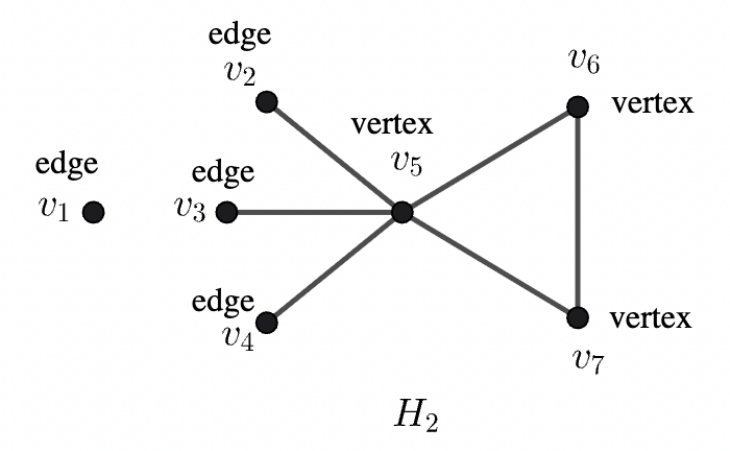}
    \end{center}
    We again may assume $uv$ or $u'v'$ appears in the copy of $H_2$. This time we need two subcases based on the position of $uv$ or $u'v'$. \\
    \\
    \textbf{Subcase 2.1.} $uv$ of $u'v'$ is the lone vertex $v_1$.\\
    \\
    \indent In this case, we get the picture below, where $w\ne u,v$ is a vertex, $wx,wy$ and $wz$ are edges incident to $w$, and $s,t\ne u,v,w,x,y,z$ are two other vertices:
    \begin{center}
        \includegraphics[width=6cm]{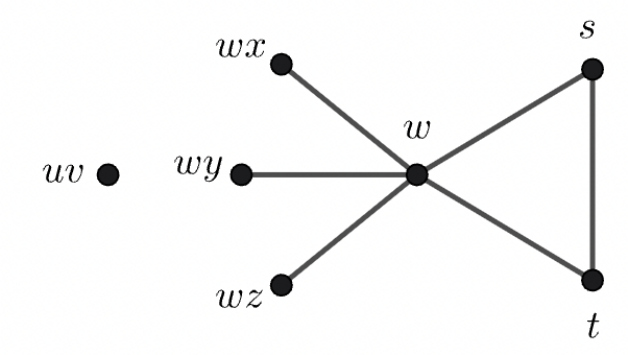}
    \end{center}
    We can start by saying that there are $\sum_{w\ne u,v}\binom{\deg_G(w)}{3}\cdot \binom{|V|-6}{2}$ ways to do this, since we can choose any three edges incident to $w$ for $wx,wy,$ and $wz$, and then $s$ and $t$ can be any two vertices that are not $u,v,w,x,y,$ or $z$. 
    
    The issue is that if one or more of $x,y,$ or $z$ is equal to $u$ or $v$, then we actually have more choices for $s$ and $t$ because fewer vertices are not allowed. In particular, if $w\sim u$ and $x=u$, then there are $\binom{\deg_G(w)-1}2$ ways to choose $y$ and $z$ and then $\binom{|V|-5}2$ ways to choose $s$ and $t$, since only the five vertices $u,v,w,y,z$ are not allowed. However, we already counted $\binom{|V|-6}2$ of these choices of $s$ and $t$ for each choice of $w,y,$ and $z$, so the number of copies of $H_2$ we need to add for each choice of $w,y,$ and $z$ is $\binom{|V|-5}2-\binom{|V|-6}2=|V|-6,$ or $\sum_{w\sim u}\binom{\deg_G(w)-1}2(|V|-6)$ in total. Similarly, we need to add $\binom{\deg_G(w)-1}2\cdot(|V|-6)$ for each $w\sim v$ to account for the extra choices of $s$ and $t$ when $x=v.$ 
    
    Finally, if $w\sim u,v$ is a common neighbor of $u$ and $v$, we could have $x=u$ and also $y=v$. In that case there are $\deg_G(w)-2$ ways to choose the final neighbor $z$ of $w$. Then the restriction is $s,t\ne u,v,w,z$, so there should be $\binom{|V|-4}2$ choices for $s$ and $t$. However, for each choice of $w$ and $y$, we have already counted $\binom{|V|-6}2+2(|V|-6)=\binom{|V|-4}2-1$ of these choices of $s$ and $t$, so there is just one additional choice of $s$ and $t$ that needs to be added. Putting all this together gives 
    \begin{align*}
        \sum_{w\ne u,v}\binom{\deg_G(w)}3 \binom{|V|-6}2 \\
        + \sum_{w\sim u}\binom{\deg_G(w)-1}2(|V|-6)+\sum_{w\sim v}\binom{\deg_G(w)-1}2(|V|-6) \\+\sum_{w\sim u,v}(\deg_G(w)-2).
    \end{align*}
    The first two lines stay the same if we swap $u$ for $u'$ and $v$ for $v'$, but the last line is greater when we use $u$ and $v$ than when we use $u'$ and $v'$, since $u$ and $v$ have at least one common neighbor but $u'$ and $v'$ do not. Thus, overall, there are more copies of $H_2$ in $\tn{sp}(G\cup uv)$ than in $\tn{sp}(G\cup u'v')$ in this case as long as there is at least one $w\sim u,v$ with $\deg_G(w)\ge 3$, and otherwise there are the same number of copies in both split graphs.\\
    \\
    \textbf{Subcase 2.2.} $uv$ or $u'v'$ is $v_2$.\\
    \\
    \indent If $v_2 = uv$, then $v_5 = u$ or $v_5 = v$, so without loss of generality assume $v_5 = u.$ Then $v_1$ must be some edge $wx$ not incident to $u$, $v_3$ and $v_4$ must be two edges $uy$ and $uz$, and $v_6$ and $v_7$ must be two other vertices $s$ and $t$:
    \begin{center}
        \includegraphics[width=6cm]{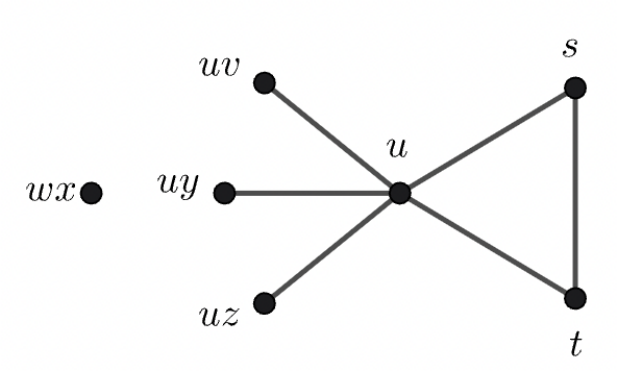}
    \end{center}
    There are $\binom{\deg_G(u)}2$ ways to choose $y$ and $z$, then $|E|-\deg_G(u)$ ways to choose $wx$ as an edge not incident to $u$, and $\binom{|V|-6}2$ ways to choose $s$ and $t$ as two vertices not equal to $u,v,w,x,y,$ or $z$. This gives $\binom{\deg_G(u)}2(|E|-\deg_G(u))\binom{|V|-6}2$ as our starting point.
    
    However, if $w$ or $x$ equals $v,y,$ or $z$, there are actually more options for $s$ and $t$, since there are fewer vertices which are not allowed. 
    
    The number of ways to have $w=y$ is $\sum_{y\sim u}(\deg_G(y)-1)(\deg_G(u)-1),$ since after choosing $y$ we must choose $wx=yx$ to be an edge incident to $u$ that is not $uy,$ and we need to choose $z$ to be a neighbor of $u$. Then there are $\binom{|V|-5}2$ choices for $s$ and $t$ since they cannot be $u,v,w=y,x,$ or $z$, but $\binom{|V|-6}2$ of those choices for $s$ and $t$ have already been counted, so we need to add $\binom{|V|-5}2-\binom{|V|-6}2 = |V|-6$ more choices for $s$ and $t$ for each choice of $u,v,w=y,x,$ and $z$. Thus in total, we need to add $\sum_{y\sim u}(\deg_G(y)-1)(\deg_G(u)-1)(|V|-6)$ to our starting sum.

    Next, for each case with $w=y$ and also $x=z$, we can choose $s$ and $t$ to be any of the $|V|-4$ vertices that are not $u,v,w,$ or $x,$ so the number of additional options for $s$ and $t$ not yet counted is $\binom{|V|-4}2 - \binom{|V|-6}2 - 2(|V|-6) = 1$, since we started with $\binom{|V|-6}2$ and then we already added $|V|-6$ once because $w=y$ and a second time because $x=z.$ Thus, we need to add 1 for each triangle $uyz = uwx$.

    Now if $w=v,$ there are $\binom{\deg_G(u)}2$ ways to choose the two neighbors $y$ and $z$ of $u$, $\deg_G(v)$ ways to choose the edge $wx = vx$, and $\binom{|V|-5}2 - \binom{|V|-6}2 = |V|-6$ ways to choose $s$ and $t$ in addition to the ones already counted. So, we need to add $\binom{\deg_G(u)}2\deg_G(v)(|V|-6)$ to our sum.

    Finally, if both $w=v$ and $y=x,$ then $y=x$ must be a common neighbor of $u$ and $v.$ There are then $\deg_G(u)-1$ ways to choose the other edge $uz$ incident to $u$, and then $\binom{|V|-4}2$ ways to choose $s$ and $t$ as two vertices different from $u,v,x,$ or $z$. Since $\binom{|V|-6}2 +2(|V|-6)$ of these choices of $s$ and $t$ have already been counted for each choice of $u,v,x,$ and $z$, the amount we need to add to our sum is $\sum_{x\sim u,v}(\deg_G(u)-1).$

    Putting all this together, the number of copies of $H_2$ in $\tn{sp}(G\cup uv)$ for Subcase 2.2 that have $v_5 = u$ is 
    \begin{align*}
        \binom{\deg_G(u)}2(|E|-\deg_G(u))\binom{|V|-6}2 \\
        + \sum_{y\sim u}(\deg_G(y)-1)(\deg_G(u)-1)(|V|-6) \\
        +(\#\tn{ of triangles containing }u\tn{ in }G) \\
        + \binom{\deg_G(u)}2\deg_G(v)(|V|-6)\\
        + \sum_{x\sim u,v}(\deg_G(u)-1).
    \end{align*}
    All of the above calculations work as long as $|V|\ge 6,$ and all of the above lines stay the same if we swap $u$ for $u'$ and $v$ for $v'$ except the last one, which is larger for $\tn{sp}(G\cup uv)$ than for $\tn{sp}(G\cup u'v')$ as long as $\deg_G(u) > 1$. Similarly, if we let $v_5 = v$ instead, our last line will become $\sum_{x\sim u,v}(\deg_G(v)-1)$, so it will be larger for $\tn{sp}(G\cup uv)$ as long as $\deg_G(v)>1$. Thus, there are more copies of $H_2$ in $\tn{sp}(G\cup uv)$ than in $\tn{sp}(G\cup u'v')$ in this subcase as long as $|V|\ge 6$ and either $\deg_G(u)>1$ or $\deg_G(v)>1$.

    \bigskip

    Combining our cases, we see that as long as $|V|\ge 6$, $u$ and $v$ have at least one common neighbor while $u'$ and $v'$ do not, and either $\deg_G(u)>1$, $\deg_G(v)>1$, or one of the common neighbors of $u$ and $v$ has degree at least 3, there will be more copies of $H_2$ in $\tn{sp}(G\cup uv)$ than in $\tn{sp}(G\cup u'v')$.
    \end{proof}

    Combining Lemmas \ref{lem:H_1} and \ref{lem:H_2} completes the proof of Theorem \ref{thm:split_graphs} and shows that $\ol{X}_{\tn{sp}(G\cup uv)}\ne \ol{X}_{\tn{sp}(G\cup u'v')}.$
\end{proof}

\printbibliography

\end{document}